\documentclass[onefignum,onetabnum]{siamart171218}

\usepackage{amsmath}
\usepackage{amsfonts}
\usepackage{amssymb}
\usepackage{graphicx}
\usepackage{epstopdf}
\usepackage{algorithmic}
\ifpdf
  \DeclareGraphicsExtensions{.eps,.pdf,.png,.jpg}
\else
  \DeclareGraphicsExtensions{.eps}
\fi

\usepackage{bm}
\usepackage{ulem}
\usepackage{lineno,hyperref}
\usepackage{stmaryrd}
\newtheorem{remark}{Remark}

\def\D{\Omega}
\def\T{\mathcal{T}}
\def\E{\widetilde{E}}

\newcommand{\rd}{{\rm d}}

\newcommand{\eps}{\epsilon}

\newcommand{\Ome}{\Omega}
\newcommand{\p}{\partial}

\def\esssupT{\underset{t\in [0,T]}{\mbox{\rm ess sup }}}

\def\supn{\underset{1\leq n \leq \ell}{\mbox{\rm sup }}}

\newcommand{\triplenorm}[1]{%
  \left\vert\kern-0.9pt\left\vert\kern-0.9pt\left\vert #1
  \right\vert\kern-0.9pt\right\vert\kern-0.9pt\right\vert}

\begin{document}

\headers{Morley element for the CH equation and the HS flow}{S. Wu and
Y. Li}
\title{Analysis of the Morley element for the Cahn-Hilliard equation
  and the Hele-Shaw flow 
\thanks{The work of Shuonan Wu is partially supported by the startup
  grant from Peking Unversity.}
}

\author{
Shuonan Wu\thanks{School of Mathematical Sciences, Peking University,
China, 100871 ({\tt snwu@math.pku.edu.cn}) }
\and
Yukun Li\thanks{Department of Mathematics, The Ohio State University,
Columbus, U.S.A. ({\tt li.7907@osu.edu})}
}

\maketitle

\begin{abstract}
The paper analyzes the Morley element method for the Cahn-Hilliard
equation. The objective is to derive the optimal error estimates and
to prove the zero-level sets of the Cahn-Hilliard equation approximate
the Hele-Shaw flow. If the piecewise $L^{\infty}(H^2)$ error bound is
derived by choosing test function directly, we cannot obtain the
optimal error order, and we cannot establish the error bound which depends
on $\frac{1}{\epsilon}$ polynomially either. To overcome this
difficulty, this paper proves them by the following steps, and the
result in each next step cannot be established without using the
result in its previous one. First, it proves some a priori estimates
of the exact solution $u$, and these regularity results are minimal to
get the main results; Second, it establishes ${L^{\infty}(L^2)}$ and
piecewise ${L^2(H^2)}$ error bounds which depend on
$\frac{1}{\epsilon}$ polynomially based on the piecewise
${L^{\infty}(H^{-1})}$ and ${L^2(H^1)}$ error bounds; Third, it
establishes piecewise ${L^{\infty}(H^2)}$ optimal error bound which
depends on $\frac{1}{\epsilon}$ polynomially based on the piecewise
${L^{\infty}(L^2)}$ and ${L^2(H^2)}$ error bounds; Finally, it proves
the ${L^\infty(L^\infty)}$ error bound and the approximation to the
Hele-Shaw flow based on the piecewise ${L^{\infty}(H^2)}$ error bound.
The nonstandard techniques are used in these steps such as the
generalized coercivity result, integration by part in space,
summation by part in time, and special properties of the Morley
elements. If one of these techniques is lacked, either we can only
obtain the sub-optimal piecewise ${L^{\infty}(H^2)}$ error order, or we
can merely obtain the error bounds which are exponentially dependent on
$\frac{1}{\epsilon}$. The approach used in this paper provides a way
to bound the errors in higher norm from the errors in lower norm step
by step, which has a profound meaning in methodology.  Numerical results
are presented to validate the optimal $L^\infty(H^2)$ error order and
the asymptotic behavior of the solutions of the Cahn-Hilliard
equation.
\end{abstract}



\section{Introduction}
Consider the following Cahn-Hilliard equation with Neumann boundary conditions:
\begin{alignat}{2}
u_t +\Delta(\epsilon\Delta u -\frac{1}{\epsilon}f(u)) &=0  &&\quad \mbox{in } \Omega_T:=\Omega\times(0,T],\label{eq20170504_1}\\
\frac{\partial u}{\partial n}
=\frac{\partial}{\partial n}(\epsilon\Delta u-\frac{1}{\epsilon}f(u))
&=0 &&\quad \mbox{on } \partial\Omega_T:=\partial\Omega\times(0,T],
\label{eq20170504_2}\\
u &=u_0 &&\quad \mbox{in } \Omega\times\{t=0\},\label{eq20170504_3}
\end{alignat}
where $\Omega\subseteq \mathbf{R}^2$ is a bounded domain, $f(u) = u^3 - u$ is the derivative of a double well potential $F(u)$ which is defined by
\begin{equation}\label{eq20170504_5}
F(u)=\frac{1}{4}(u^2-1)^2.
\end{equation}
The Allen-Cahn equation \cite{allen1979microscopic, bartels2011robust,
chen1994spectrum, feng2003numerical, feng2014finite,
feng2014analysis, feng2017finite, ilmanen1993convergence}
and the Cahn-Hilliard equation
\cite{alikakos1994convergence,chen1994spectrum, kovacs2011finite,
wu2017multiphase} are
two basic phase field models to describe the phase transition process.
They are also proved to be related to geometric flow. For example, the
zero-level sets of the Allen-Cahn equation approximate the mean
curvature \cite{evans1992phase, ilmanen1993convergence} and the
zero-level sets of the Cahn-Hilliard equation approximate the
Hele-Shaw flow \cite{stoth1996convergence, alikakos1994convergence}.
The Cahn-Hilliard equation was introduced by J. Cahn and J. Hilliard
in \cite{cahn1958free} to describe the process of phase separation,
   by which the two components of a binary fluid separate and form
   domains pure in each component.  It can be interpreted as the
   $H^{-1}$ gradient flow \cite{alikakos1994convergence} of the
   Cahn-Hilliard energy functional
\begin{align}\label{eq2.1}
J_\epsilon(v):= \int_\Omega \Bigl( \frac\eps2 |\nabla v|^2+
    \frac{1}{\epsilon} F(v) \Bigr)\, \rd x.
\end{align}
There are a few papers
\cite{aristotelous2013mixed,
xu2016stability,du1991numerical,elliott1989nonconforming} discussing
the error bounds, which depend on the exponential power of
$\frac{1}{\epsilon}$, of the numerical methods for Cahn-Hilliard
equation.  Such an estimate is clearly not useful for small
$\epsilon$, in particular, in addressing the issue whether the
computed numerical interfaces converge to the original sharp
interface of the Hele-Shaw problem.  Instead, the polynomial
dependence in $\frac{1}{\epsilon}$ is proved in
\cite{feng2004error, feng2005numerical} using the standard finite element
method, and in \cite{feng2016analysis,li2015numerical} using the
discontinuous Galerkin method. Due to the high efficiency of the
Morley elements, compared with mixed finite element methods or
$C^1$-conforming finite element methods, the Morley finite element
method is used to derive the error bound which depends on
$\frac{1}{\epsilon}$ polynomially in this paper.

The highlights of this paper are fourfold. First, it establishes the
piecewise ${L^{\infty}(L^2)}$ and ${L^2(H^2)}$ error bounds which
depend on $\frac{1}{\epsilon}$ polynomially. If the standard technique
is used, we can only prove that the error bounds depend on
$\frac{1}{\epsilon}$ exponentially, which can not be used to
prove our main theorem. To prove these bounds, special properties of
the Morley elements are explored, i.e., Lemma 2.3 in
\cite{elliott1989nonconforming}, and piecewise
${L^{\infty}(H^{-1})}$ and ${L^2(H^1)}$ error
bounds \cite{li2017error} are required. Second, by making use of the
piecewise ${L^{\infty}(L^2)}$ and ${L^2(H^2)}$ error bounds above, it
establishes the piecewise ${L^{\infty}(H^2)}$ error bound which
depends on $\frac{1}{\epsilon}$ polynomially. If the standard
technique is used, we can only get the error bound in Remark
\ref{rmk20180823_1}, which does not have an optimal order. The crux
here is to employ the summation by part in time and integration by
part in space techniques simultaneously to handle the nonlinear term,
together with the special properties of the Morley elements.
Third, the minimal regularity of $u$ is used, i.e.,
$\|u_{tt}\|_{L^2(L^2)}$ regularity instead of
$\|u_{tt}\|_{L^\infty(L^2)}$ regularity is used, and the a priori
estimate is derived in Theorem \ref{thm20180609_1}. Fourth, the
${L^\infty(L^\infty)}$ error bound is established using the optimal
piecewise ${L^{\infty}(H^2)}$ error, by which the main result that the
zero-level sets of the Cahn-Hilliard equation approximate the
Hele-Shaw flow is proved in Section \ref{sec5}.

The organization of this paper is as follows. In Section \ref{sec2},
the standard Sobolev space notation is introduced, some
useful lemmas are stated, and a new a priori estimate of the exact
solution $u$ is derived. In Section \ref{sec3}, the fully discrete
approximation based on the Morley finite element space is presented. In
Section \ref{sec4}, first the polynomially dependent piecewise
${L^{\infty}(L^2)}$ and ${L^2(H^2)}$ error bounds are established
based on piecewise ${L^{\infty}(H^{-1})}$ and ${L^2(H^1)}$ error
bounds, then the polynomially dependent piecewise ${L^{\infty}(H^2)}$
error bound is established based on piecewise ${L^{\infty}(L^2)}$ and
${L^2(H^2)}$ error bounds, by which the ${L^\infty(L^\infty)}$ error
bound is proved. In Section \ref{sec5}, the approximation of the
zero-level sets of the Cahn-Hilliard equation of the Hele-Shaw flow is
proved. In Section \ref{sec6}, numerical tests are presented to
validate our theoretical results, including the optimal error orders
and the approximation of the Hele-Shaw flow.

\section{Preliminaries}\label{sec2}
In this section, we present some results which will be used in
the following sections. Throughout this paper, $C$ denotes a generic
positive constant which is independent of interfacial length
$\epsilon$, spacial size $h$, and time step size $k$, and it may have
different values in different formulas. The standard Sobolev space
notation below is used in this paper. 

\begin{alignat*}{2}
\|v\|_{0,p,A}&=\bigg(\int_{A}|v|^p\,\rd x\bigg)^{1\slash p}\qquad &&1\le
p<\infty,\\
\|v\|_{0,\infty,A}&=\underset{A}{\mbox{\rm ess sup }} |v|,\\
|v|_{m,p,A}&=\bigg(\sum_{|\alpha|=m}\|D^{\alpha}v\|_{0,p,A}^p\bigg)^{1\slash
p}\qquad &&1\le p<\infty,\\
\|v\|_{m,p,A}&=\bigg(\sum_{j=0}^m|v|_{m,p,A}^p\bigg)^{1\slash p}.
\end{alignat*}

Here $A$ denotes some domain, i.e., a single mesh element $K$ or the
whole domain $\Omega$. When $A=\Omega$, $\|\cdot\|_{H^k},
\|\cdot\|_{L^k}$ are used to denote $\|\cdot\|_{H^k(\Omega)},
\|\cdot\|_{L^k(\Omega)}$ respectively, and $\|\cdot\|_{0,2}$ is also used to denote $\|\cdot\|_{L^2(\Omega)}$. Let $\T_h$ be
a family of quasi-uniform triangulations of domain $\Omega$, and
$\mathcal{E}_h$ be a collection of edges, then the global mesh
dependent semi-norm, norm and inner product are defined below
\begin{align*}
|v|_{j,p,h}&=\bigg(\sum_{K\in\T_h}|v|_{j,p,K}^p\bigg)^{1\slash p},\\
\|v\|_{j,p,h}&=\bigg(\sum_{K\in\T_h}\|v\|_{j,p,K}^p\bigg)^{1\slash p},\\
(w,v)_h&=\sum_{K\in\T_h}\int_Kw(x)v(x)\,\rd x.
\end{align*}

Define $L^2_0(\Ome)$ as the mean zero functions in $L^2(\Ome)$.  For
$\Phi\in L_0^2(\Ome)$, let $u := -\Delta^{-1}\Phi \in H^1(\Ome)\cap
L^2_0(\Ome)$ such that
\begin{alignat}{2}
-\Delta u &= \Phi&&\qquad \mathrm{in}\ \Omega,\notag\\
\frac{\partial u}{\partial n}&= 0&&\qquad \mathrm{on}\
  \partial\Omega.\notag
\end{alignat}
Then we have
\begin{align}\label{eq6_add}
-(\nabla\Delta^{-1}\Phi,\nabla v) = (\Phi,v)\quad \mathrm{in}\
    \Omega\qquad\forall v\in H^1(\Omega)\cap
L^2_0(\Ome).
\end{align}
For $v\in L^2_0(\Ome)$ and $\Phi\in L^2_0(\Ome)$, define the
continuous $H^{-1}$ inner product by
\begin{align}\label{eq7_add}
(\Phi, v)_{H^{-1}} := (\nabla\Delta^{-1}\Phi,\nabla\Delta^{-1}v) =
(\Phi,-\Delta^{-1}v) = (v,-\Delta^{-1}\Phi).
\end{align}

As in \cite{chen1994spectrum, feng2016analysis, feng2004error,
feng2005numerical, li2015numerical, li2017error},
we made the following assumptions on the initial condition.  These
assumptions were used to derive the a priori estimates for the
solution of problem \eqref{eq20170504_1}--\eqref{eq20170504_5}.

{\bf General Assumption} (GA)
\begin{itemize}
\item[(1)] Assume that $m_0\in (-1,1)$ where
\begin{align*}
m_0:=\frac{1}{|\Omega|}\int_{\Omega}u_0(x)\,\rd x. 
\end{align*}
\item[(2)] There exists a nonnegative constant $\sigma_1$ such that
\begin{align*}
J_{\epsilon}(u_0)\leq C\epsilon^{-2\sigma_1}.
\end{align*}
\item[(3)]
There exist nonnegative constants $\sigma_2$, $\sigma_3$ and $\sigma_4$ such that
\begin{align*}
\big\|-\epsilon\Delta u_0 +\epsilon^{-1} f(u_0)\big\|_{H^{\ell}} \leq
C\epsilon^{-\sigma_{2+\ell}}\qquad \ell=0,1,2.
\end{align*}
\end{itemize}

Under the above assumptions, the following a priori estimates of the
solution were proved in
\cite{feng2016analysis,feng2004error, feng2005numerical, li2015numerical}.

\begin{theorem}\label{prop2.1}
The solution $u$ of problem \eqref{eq20170504_1}--\eqref{eq20170504_5}
satisfies the following energy estimate:
\begin{align}
&\esssupT  \Bigl( \frac{\epsilon}{2}\|\nabla u\|_{L^2}^2
    +\frac{1}{\epsilon}\|F(u)\|_{L^1} \Bigr)
  +\int_{0}^{T}\|u_t(s)\|_{H^{-1}}^2\, \rd s
\leq J_{\epsilon}(u_0)\label{eq2.5}.
\end{align}
Moreover, suppose that GA (1)--(3) hold, $u_0\in H^4(\Omega)$ and
$\p\Omega\in C^{2,1}$, then $u$ satisfies the additional estimates:
\begin{align}
&\frac{1}{|\Omega|}\int_{\Omega}u(x,t)\, \rd x=m_0 \quad\forall t\geq 0,
\label{eq2.8}\\
%
&\esssupT\|\Delta u\|_{L^2}\leq
C\epsilon^{-\max\{\sigma_1+\frac{5}{2},\sigma_3+1\}},\label{eq2.12}\\
&\esssupT\|\nabla\Delta u\|_{L^2}\leq
C\epsilon^{-\max\{\sigma_1+\frac{5}{2},\sigma_3+1\}},\label{eq2.13_add}\\
%
&\epsilon\int_0^{T}\|\Delta u_t\|_{L^2}^2\,\rd s+\esssupT\|u_t\|_{L^2}^2
\leq
C\epsilon^{-\max\{2\sigma_1+\frac{13}{2},2\sigma_3+\frac{7}{2},2\sigma_2+4,2\sigma_4\}}.\label{eq2.15_add}
\end{align}
Furthermore, if there exists $\sigma_5>0$ such that
\begin{equation}\label{eq2.17}
\mathop{\rm{lim}}_{s\rightarrow0^{+}}\limits\|\nabla
u_t(s)\|_{L^2}\leq C\epsilon^{-\sigma_5},
\end{equation}
then there hold
\begin{align}
&\esssupT\|\nabla u_t\|_{L^2}^2 + \epsilon\int_0^{T}\|\nabla\Delta u_t\|_{L^2}^2\,\rd s 
\leq C\rho_0(\epsilon),\label{eq2.18}\\
&\int_0^{T}\|u_{tt}\|_{H^{-1}}^2\,\rd s \leq C\rho_1(\epsilon),\label{eq2.19}
%
\end{align}
where
\begin{align*}
\rho_0(\epsilon)
  &:=\epsilon^{-\frac{1}{2}\max\{2\sigma_1+5,2\sigma_3+2\}
    -\max\{2\sigma_1+\frac{13}{2},2\sigma_3+\frac{7}{2},2\sigma_2+4\}}
    +\epsilon^{-2\sigma_5}\\
&\qquad 
+\epsilon^{-\max\{2\sigma_1+7,2\sigma_3+4\}},\\
\rho_1(\epsilon) &:=\epsilon \rho_0(\epsilon).
\end{align*}
\end{theorem}

Besides, an extra a priori estimates of solution $u$ is needed in
this paper.
\begin{theorem}\label{thm20180609_1}
Under the assumptions of Theorem \ref{prop2.1} and if there exists
$\sigma_6>0$ such that
\begin{align}\label{eq20180606_5}
\|\Delta u_t(0)\|_{L^2}\le C\epsilon^{-\sigma_6},
\end{align}
then there hold
\begin{align}\label{eq20180606_1}
\esssupT\|\Delta u_t\|_{L^2}^2 + \epsilon\int_0^{T}\|\Delta^2
u_t\|_{L^2}^2\,\rd s &\leq C \rho_2(\epsilon),\\
\esssupT\epsilon\|\Delta u_t\|_{L^2}^2+\int_0^{T}\|u_{tt}\|_{L^2}^2\,\rd s
&\leq C \rho_3(\epsilon),\label{eq20180606_2}
\end{align}
where
\begin{align*}
\rho_2(\epsilon)&:=\epsilon^{-\max\{2\sigma_1+\frac{13}{2},2\sigma_3+\frac{7}{2},2\sigma_2+4,2\sigma_4\}
  - \max\{2\sigma_1+5, 2\sigma_3+2\} - 3} \\ 
& \qquad + \epsilon^{-\max\{\sigma_1+\frac52,\sigma_3+1\}-3}\rho_0(\epsilon) +
\epsilon^{-2\sigma_6},\\
\rho_3(\epsilon)&:=\epsilon\rho_2(\epsilon).
\end{align*}
\end{theorem}
\begin{proof} 
Using the Gagliardo-Nirenberg inequalities \cite{adams2003sobolev}
in two-dimensional space, we have 
\begin{align}\label{eq20180801_3}
\|\nabla u\|_{L^{\infty}}\leq C\bigg(\|\nabla\Delta
u\|_{L^2}^{\frac12}
\|u\|_{L^{\infty}}^{\frac12}+\|u\|_{L^{\infty}}\bigg)\le
C \epsilon^{-\frac12 \max\{\sigma_1+\frac52, \sigma_3 + 1\}}.
\end{align}
Since $f'(u) = 3u^2 - 1$, using Sobolev embedding
theorem \cite{adams2003sobolev}, \eqref{eq2.5}, \eqref{eq2.12}, \eqref{eq2.13_add}, \eqref{eq2.15_add} and \eqref{eq2.18}, we have 
\begin{align} \label{eq20180731_1}
&~\quad \int_0^T \|\Delta(f'(u)u_t)\|_{L^2}^2 \,\rd s \\
&= \int_0^T \|6uu_t \Delta u + 12 u\nabla u \cdot \nabla u_t +
6u_t\nabla u \cdot \nabla u + (3u^2 -1) \Delta u_t\|_{L^2}^2 \,\rd s \notag
\\
&\le C\int_0^T\|\Delta u\|_{L^2}^2 \|u_t\|_{L^\infty}^2\, \rd s 
+ C\int_0^T \|\nabla u\|_{L^\infty}^2 \|\nabla u_t\|_{L^2}^2 \,\rd s \notag \\
& ~\quad + C\int_0^T \|\nabla u\|_{L^\infty}^{4} \|u_t\|_{L^2}^2 \,\rd s 
+ C\int_0^T \|\Delta u_t\|_{L^2}^2\, \rd s \notag \\
&\le C\|\Delta u\|_{L^\infty(L^2)}^2 \int_0^T\|u_t\|_{H^2}^2\, \rd s 
+ C\|\nabla u_t\|_{L^\infty(L^2)}^2 \|\nabla
u\|_{L^\infty(L^\infty)}^2 \notag \\
& ~\quad + C\|\nabla u\|_{L^\infty(L^\infty)}^{4}
\|u_t\|_{L^\infty(L^2)}^2 
+ C\int_0^T \|\Delta u_t\|_{L^2}^2\,\rd s \notag \\
&\le C \epsilon^{-\max\{2\sigma_1+\frac{13}{2},2\sigma_3+\frac{7}{2},
  2\sigma_2+4,2\sigma_4\} - \max\{2\sigma_1+5, 2\sigma_3+2\}-1} \notag \\ 
&~\quad + C\epsilon^{-\max\{\sigma_1 + \frac52,\sigma_3+1\}}\rho_0(\epsilon)
  \notag \\
& ~\quad + C \epsilon^{-\max\{2\sigma_1+\frac{13}{2},2\sigma_3+\frac{7}{2},
  2\sigma_2+4,2\sigma_4\} - \max\{2\sigma_1+5, 2\sigma_3 + 2\}} \notag \\
& ~\quad + C \epsilon^{-\max\{2\sigma_1+\frac{13}{2},2\sigma_3+\frac{7}{2},
  2\sigma_2+4,2\sigma_4\} - 1} \notag \\
&\le C \epsilon^{-\max\{2\sigma_1+\frac{13}{2},2\sigma_3+\frac{7}{2},
  2\sigma_2+4,2\sigma_4\} - \max\{2\sigma_1+5,
  2\sigma_3+2\}} \notag \\ 
&~\quad 
+ C\epsilon^{-\max\{\sigma_1 + \frac52, \sigma_3+1\}}\rho_0(\epsilon)
  \notag.
\end{align}

Taking the derivative with respect to $t$ on both sides of
\eqref{eq20170504_1}, we get
\begin{align}\label{eq20180606_3}
u_{tt}+\epsilon\Delta^2u_t-\frac{1}{\epsilon}\Delta(f'(u)u_t)=0.
\end{align}
Testing \eqref{eq20180606_3} with $\Delta^2u_t$, and taking the
integral over $(0,T)$, we obtain
\begin{align}\label{eq20180606_4}
&~\quad\frac{1}{2}\|\Delta
u_t(T)\|_{L^2}^2+\epsilon\int_0^{T}\|\Delta^2u_t\|_{L^2}^2\,\rd s \\ 
& =\frac{1}{\epsilon}\int_0^{T}(\Delta(f'(u)u_t),\Delta^2 u_t)\, \rd s +
\frac{1}{2}\|\Delta u_t(0)\|_{L^2}^2 \notag \\
&\le\frac{C}{\epsilon^3}\int_0^{T} \|\Delta(f'(u) u_t)\|_{L^2}^2\, \rd s 
+\frac{\epsilon}{2}\int_0^{T}\|\Delta^2u_t\|_{L^2}^2\,\rd s 
+C\epsilon^{-2\sigma_6}\notag.
\end{align}
Then \eqref{eq20180606_1} is obtained by \eqref{eq20180731_1}.

Next we bound \eqref{eq20180606_2}. Testing \eqref{eq20180606_3} with
$u_{tt}$, taking the integral over $(0,T)$, and using
\eqref{eq20180606_4}, we obtain
\begin{align}\label{eq20180606_8}
&~\quad \int_0^{T}\|u_{tt}\|_{L^2}^2\,\rd s + \frac{\epsilon}{2}\|\Delta
u_t(T)\|_{L^2}^2 \\ 
&\le \frac{\epsilon}{2}\|\Delta u_t(0)\|_{L^2}^2 +
\frac{C}{\epsilon^2}\int_0^{T} \|\Delta(f'(u)u_t)\|_{L^2}^2\,\rd s
+ \frac{1}{2}\int_0^{T}\|u_{tt}\|_{L^2}^2\,\rd s \notag.
\end{align}
Then \eqref{eq20180606_2} is obtained by \eqref{eq20180731_1}.
\end{proof}

The next lemma gives an $\epsilon$-independent lower bound for the
principal eigenvalue of the linearized Cahn-Hilliard operator
$\mathcal{L}_{CH}$ defined below. The proof of this lemma can be found
in \cite{chen1994spectrum}. 
\begin{lemma}\label{lem3.4}
Suppose that GA (1)--(3) hold. Given a smooth initial curve/surface
$\Gamma_0$, let $u_0$ be a smooth function satisfying $\Gamma_0 =
\{x\in\Omega; u_0(x)=0\}$ and some profile described in \cite{chen1994spectrum}.
Let $u$ be the solution to problem
\eqref{eq20170504_1}--\eqref{eq20170504_5}.  Define $\mathcal{L}_{CH}$
as
\begin{equation*}
\mathcal{L}_{CH} := \Delta\left(\eps\Delta-\frac{1}{\eps}f'(u)I\right).
\end{equation*}
Then there exists $0<\epsilon_0\ll 1$ and a positive constant $C_0$ such
that the principle eigenvalue of the linearized Cahn-Hilliard operator
$\mathcal{L}_{CH}$ satisfies 
\begin{equation*}
\lambda_{CH}:=\mathop{\inf}_{\substack{0\neq\psi\in H^1(\Omega)\\ \Delta w=\psi}}
\limits\frac{\epsilon\|\nabla\psi\|_{L^2}^2+\frac{1}{\epsilon}(f'(u)\psi,\psi)}{\|\nabla w\|_{L^2}^2}\geq -C_0
\end{equation*}
for $t\in [0,T]$ and $\eps\in (0,\eps_0)$.
\end{lemma}

\section{Fully Discrete Approximation}\label{sec3}
In this section, the backward Euler is used for time stepping,
and the Morley finite element discretization is used for space
discretization. 

\subsection{Morley finite element space}

Define the Morley finite element spaces $S^h$ below
\cite{brenner1999convergence,brenner2013morley,elliott1989nonconforming}:
\begin{align*}
S^h := \{& v_h\in L^{\infty}(\D): v_h\in P_2(K), v_h ~\text{is
continuous at the vertices of all triangles,} \\ 
&\frac{\partial v_h}{\partial n} \text{ is continuous at the midpoints
of interelement edges of triangles} \}.
\end{align*}

We use the following notation
\begin{equation*}
H^j_E(\Omega):=\{v\in H^j(\Omega): \frac{\partial v}{\partial
  n}=0~\text{on}~\partial\Omega\}\qquad j=1, 2, 3.
\end{equation*}
Corresponding to $H^j_E(\Omega)$, define $S^h_E$ as a subspace of
$S^h$ below: 
\begin{equation*}
S^h_E := \{v_h\in S^h: \frac{\partial v_h}{\partial n}=0 \text{ at the
midpoints of the edges on } \partial\Omega\}.
\end{equation*}
We also define $\mathring{H}_E^j(\Omega) = H_E^j(\Omega) \cap
L_0^2(\Omega), j=1,2,3$, and $\mathring{S}^h_E = S^h_E \cap
L_0^2(\Omega)$, where $L_0^2(\Omega)$ denotes the set of mean zero functions.

The enriching operator $\E_h$ is restated
\cite{brenner1996two,brenner1999convergence,brenner2013morley}. Let
$\widetilde{S}_E^h$ be the Hsieh-Clough-Tocher macro element space,
which is an enriched space of the Morley finite element space $S_E^h$.
Let $p$ and $m$ be the internal vertices and midpoints of triangles
$\T_h$. Define $\E_h: S_E^h\rightarrow \widetilde{S}_E^h$ by
\begin{align*}
(\E_h v)(p) &= v(p),\\
\frac{\p (\E_h v)}{\p n}(m) &= \frac{\p v}{\p n}(m),\\
(\p^{\beta}(\E_h v))(p) &= \text{average of } (\p^{\beta}v_i)(p)\qquad |\beta|=1,
\end{align*}
where $v_i=v|_{T_i}$ and triangle $T_i$ contains $p$ as a vertex.

Define the interpolation operator $I_h: H^2_E(\Omega)\rightarrow
S_E^h$ such that
\begin{align*}
(I_h v)(p)&=v(p),\\ 
\frac{\p (I_h v)}{\p n}(m)&=\frac{1}{|e|}\int_e\frac{\p v}{\p n}\,\rd S,
\end{align*}
where $p$ ranges over the internal vertices of all the triangles $T$,
and $m$ ranges over the midpoints of all the edges $e$.  It can be
proved that
\cite{brenner1996two,brenner1999convergence,brenner2013morley,elliott1989nonconforming}
\begin{alignat}{2}\label{eq20170812_6}
|v-I_hv|_{j,p,K}&\le Ch^{3-j}|v|_{3,p,K}\qquad&&\forall
K\in\mathcal{T}_h,\quad\forall v\in H^3(K),\quad j=0,1,2,\\
\|\E_h v-v\|_{j,2,h}&\le Ch^{2-j}|v|_{2,2,h}\quad&&\forall v\in
S_E^h,\quad j=0,1,2.\label{eq20171006_1}
\end{alignat}

Notice that $\E_h$ and $I_h$ cannot preserve the mean zero functions.
Let $\mathring{\widetilde S_E^h}:= \widetilde{S}_E^h \cap
L_0^2(\Omega)$.  Define $\mathring{\E_h}:\mathring{S}_E^h \mapsto
\mathring{\widetilde S_E^h}$ such that 
\begin{align} \label{eq20180524_1} 
\mathring{\E_h}v = \E_h v - \frac{1}{|\Omega|}\int_{\Omega} \E_h v \,\rd x.
\end{align} 
Using \eqref{eq20171006_1}, we have 
$$ 
\int_\Omega \E_h v \,\rd x = (\E_h v - v, 1) \leq |\Omega|^{1/2}\|\E_h v -
v\|_{0,2} \leq Ch^2|v|_{2,2,h} \qquad \forall v \in \mathring{S}_E^h.
$$ 
Then
\begin{align} \label{eq20180524_2}
\|\mathring{\E_h} v-v\|_{j,2,h}&\le Ch^{2-j}|v|_{2,2,h}\qquad\forall v\in
\mathring{S}_E^h,\quad j=0,1,2.
\end{align}

Finally the following spaces are needed
\begin{alignat*}{2}
&H^{3,h}(\Omega)=S^h\oplus  H^3(\Omega), &&\qquad H_E^{3,h}(\Omega)=S_E^h\oplus  H_E^3(\Omega),\\
&H^{2,h}(\Omega)=S^h\oplus  H^2(\Omega), &&\qquad H_E^{2,h}(\Omega)=S_E^h\oplus  H_E^2(\Omega),\\
&H^{1,h}(\Omega)=S^h\oplus  H^1(\Omega), &&\qquad H_E^{1,h}(\Omega)=S_E^h\oplus  H_E^1(\Omega),
\end{alignat*}
where, for instance,
\begin{align*}
S_E^h\oplus  H_E^3(\Omega):=\{u+v: u\in S_E^h\ \ \text{and}\ \ v\in H_E^3(\Omega)\}.
\end{align*}

\subsection{Formulation}
The weak form of \eqref{eq20170504_1}--\eqref{eq20170504_5} is to seek
$u(\cdot,t)\in H^2_E(\D)$ such that
\begin{align}\label{eq20180211_1}
(u_t,v)+\epsilon a(u,v) +\frac{1}{\epsilon}(\nabla f(u), \nabla
    v)&= 0\quad\forall v\in H_E^2(\D),\\
u(\cdot,0)&=u_0\in H_E^2(\D),\label{eq20180211_2}
\end{align}
where the bilinear form $a(\cdot,\cdot)$ is defined as
\begin{align}\label{eq20170504_8}
a(u,v):=\int_{\D}\Delta u\Delta v+\bigl(\frac{\partial^2u}{\partial x\partial y}\frac{\partial^2v}{\partial x\partial y}-\frac12\frac{\partial^2u}{\partial x^2}\frac{\partial^2v}{\partial y^2}-\frac12\frac{\partial^2u}{\partial y^2}\frac{\partial^2v}{\partial x^2}\bigr)\,\rd x \rd y
\end{align}
with Poisson's ratio $\frac12$.


Next define the discrete bilinear form
\begin{align}\label{eq20170504_9}
a_h(u,v)&:=\sum_{K\in\mathcal{T}_h}\int_K\Delta u\Delta v+\bigl(\frac{\partial^2u}{\partial x\partial y}\frac{\partial^2v}{\partial x\partial y}-\frac12\frac{\partial^2u}{\partial x^2}\frac{\partial^2v}{\partial y^2}-\frac12\frac{\partial^2u}{\partial y^2}\frac{\partial^2v}{\partial x^2}\bigr)
\,\rd x\rd y.
\end{align}

Based on the bilinear form \eqref{eq20170504_9}, a fully discrete Galerkin method is to seek $u_h^n\in S^h_E$ such that 
\begin{align}\label{eq20170504_11}
(d_tu_h^{n},v_h)+\epsilon a_h(u_h^{n},v_h)+\frac{1}{\epsilon}(\nabla f(u_h^{n}),\nabla v_h)_h&=0\quad\forall v_h\in S^h_E,\\
u_h^0&=u_0^h\in S^h_E,\label{eq20170504_12}
\end{align}
where the difference operator $d_tu_h^{n} :=
\frac{u_h^{n}-u_h^{n-1}}{k}$ and $u_0^h := P_hu(t_0)$, where the
operator $P_h$ is defined below. 

\subsection{Elliptic operator $P_h$}
We define
\begin{align*}
R:=\bigl\{v\in H_E^2(\D): \Delta v\in H_E^2(\D)\bigr\}.
\end{align*}
Then $\forall v\in R$, define the elliptic operator $P_h$ (cf.
\cite{elliott1989nonconforming}) by seeking $P_hv\in S_E^h$ such that
\begin{align}\label{eq20170504_14}
\tilde b_h(P_hv,w):=(\epsilon\Delta^2v-\frac{1}{\epsilon}\nabla \cdot
(f'(u)\nabla v)+\alpha v,w)\qquad\forall w\in S_E^h,
\end{align}
where
\begin{align}\label{eq20170504_15}
\tilde b_h(v,w):=\epsilon a_h(v,w)+\frac{1}{\epsilon}(f'(u)\nabla
    v,\nabla w)_h+\alpha(v,w),
\end{align}
and $\alpha$ should be chosen as $\alpha = \alpha_0 \epsilon^{-3}$ to
guarantee the coercivity of $\tilde{b}_h(\cdot,\cdot)$.  More
precisely, first we cite some lemmas in
\cite{elliott1989nonconforming}, which will be used in this paper.

\begin{lemma}[Lemma 2.3 in \cite{elliott1989nonconforming}]
\label{lem20180817_1}
Let $w,z \in H_E^{2,h}(\Omega)$, then 
$$
\left| \sum_{K \in \mathcal{T}_h} \int_{\partial K} \frac{\partial
w}{\partial n}z \,\rd S \right| 
\leq Ch(h\|w\|_{2,2,h}\|z\|_{2,2,h} + \|w\|_{1,2,h}\|z\|_{2,2,h} +
\|w\|_{2,2,h}\|z\|_{1,2,h}).
$$
\end{lemma}

\begin{lemma}[Lemma 2.5 in \cite{elliott1989nonconforming}]
\label{lem20180817_2}
Let $z\in H^{2,h}(\Omega)$ and $w\in H_E^2(\Omega) \cap H^3(\Omega)$,
and define $B_h(w,z)$ by 
$$ 
B_h(w,z) = \sum_{K\in \mathcal{T}_h} \int_{\partial K} \left( 
\Delta w \frac{\partial z}{\partial n} + \frac{1}{2} \frac{\partial^2
w}{\partial n \partial s} - \frac{1}{2}\frac{\partial^2 w}{\partial
s^2}\frac{\partial z}{\partial n}
\right)\,\rd S,
$$ 
then we have 
\begin{equation} \label{eq20180817_2} 
|B_h(w,z)| \leq Ch |w|_{3,2,h}|z|_{2,2,h}.
\end{equation} 
\end{lemma}

For any $w\in S_E^h$, using Lemma \ref{lem20180817_1} and the inverse inequality, we have  
$$ 
\begin{aligned}
|w|_{1,2,h}^2 & \leq |w|_{2,2,h}\|w\|_{0,2} + \left| \sum_{K \in
\mathcal{T}_h} \int_{\partial K} \frac{\partial w}{\partial n}z \,\rd S
\right| \leq C \|w\|_{2,2,h}\|w\|_{0,2} \\ 
&\leq C ( |w|_{2,2,h}\|w\|_{0,2} + |w|_{1,2,h}\|w\|_{0,2} +
\|w\|_{0,2}^2 ). 
\end{aligned}
$$ 
The kick-back argument gives 
$$ 
|w|_{1,2,h}^2 \leq C ( |w|_{2,2,h}\|w\|_{0,2} + \|w\|_{0,2}^2 ).
$$ 
Hence,
\begin{align}
\label{eq20180817_3}
\tilde{b}_h(w,w) &= \epsilon a_h(w,w) + \frac{1}{\epsilon} (f'(u)
    \nabla w, \nabla w) + \frac{\alpha_0}{\epsilon^3}(w,w) \\
& \geq \frac{1}{\epsilon^3} \left( 
\frac{\epsilon^4}{2}|w|_{2,2,h}^2  - C\epsilon^2|w|_{1,2,h}^2 +
\alpha_0\|w\|_{0,2}^2 \right) \notag \\
& \geq \frac{1}{\epsilon^3} \left( 
\frac{\epsilon^4}{4}|w|_{2,2,h}^2 +
(\alpha_0 - C)\|w\|_{0,2}^2 \right) \notag,
\end{align}
which implies the coercivity of $\tilde{b}_h(\cdot,\cdot)$ when
$\alpha_0$ is large enough but independent of $\epsilon$. 

Next we give the properties of $P_h$.  Define $b_h(\cdot,\cdot) :=
\epsilon^3 \tilde{b}_h(\cdot,\cdot)$ and a norm 
$$ 
\triplenorm{v}_{2,2,h}^2 := \epsilon^4 |v|_{2,2,h}^2 +
\epsilon^2|v|_{1,2,h}^2 + \|v\|_{0,2}^2,\qquad 
$$ 

\begin{lemma} \label{lem20180817_3}
Consider the following problems: 
\begin{align}
b_h(v, \eta) &= F_h(\eta) \quad \forall \eta \in H_E^2(\Omega),
  \label{eq20180817_4} \\
b_h(v_h, \chi) &= \widetilde{F}_h(\chi) \quad \forall \chi \in S_E^h.
  \label{eq20180817_5}
\end{align}
Then we have 
\begin{align}
& \quad ~\triplenorm{v - v_h}_{2,2,h} \label{eq20180817_6} \\ 
& \leq Ch\left\{ (\epsilon+h)^2|v|_{3,2} +
|v|_{1,2} + \sup_{\chi \in S_E^h} \frac{F_h(\E_h\chi) -
\widetilde{F}_h(\chi) + \alpha_0(v, \chi -
\E_h\chi)}{\triplenorm{\chi}_{2,2,h}} \right\}. \notag 
\end{align}
\end{lemma}
\begin{proof}
Using \eqref{eq20180817_3} and the Strang Lemma, we have 
$$ 
\begin{aligned}
&\quad~ \triplenorm{v - v_h}_{2,2,h} \\ 
& \leq C \left( \inf_{\psi \in S_E^h}\triplenorm{v -
\psi}_{2,2,h} + \sup_{\chi \in S_E^h} \frac{b_h(v, \chi) -
\widetilde{F}_h(\chi)}{\triplenorm{\chi}_{2,2,h}} \right) \\
& \leq 
 C \left( \inf_{\psi \in S_E^h}\triplenorm{v -
\psi}_{2,2,h} + \sup_{\chi \in S_E^h}
\frac{b_h(v, \chi - \E_h \chi) + b_h(v,\E_h\chi)-
\widetilde{F}_h(\chi)}{\triplenorm{\chi}_{2,2,h}} \right) \\
& \leq 
 C \left( \inf_{\psi \in S_E^h}\triplenorm{v -
\psi}_{2,2,h} + \sup_{\chi \in S_E^h}
\frac{b_h(v, \chi - \E_h \chi) + F_h(\E_h\chi)-
\widetilde{F}_h(\chi)}{\triplenorm{\chi}_{2,2,h}} 
\right).
\end{aligned}
$$ 
Using Lemma \ref{lem20180817_2} and \eqref{eq20171006_1}, we have 
$$ 
\begin{aligned}
b_h(v, \chi - \E_h\chi) &= \epsilon^4 a_h(v, \chi - \E_h\chi) +
\epsilon^2 (f'(u)\nabla v, \nabla(\chi - \E_h\chi)) + (\alpha_0v, \chi
- \E_h\chi) \\
& \leq Ch\left( 
\epsilon^4 |v|_{3,2}|\chi|_{2,2,h} + \epsilon^2|v|_{1,2}|\chi|_{2,2,h}
\right) + (\alpha_0 v, \chi - \E_h\chi) \\
& \leq Ch\left( \epsilon^2 |v|_{3,2} + |v|_{1,2} \right) 
\triplenorm{\chi}_{2,2,h}
+ (\alpha_0 v, \chi - \E_h\chi) \\
\end{aligned}
$$ 
Then we obtain the desired bound \eqref{eq20180817_6} by
the approximation properties of Morley interpolation operator
\eqref{eq20170812_6}.
\end{proof}

\begin{theorem} \label{thm20180817_1}
Suppose $u$ solves the Cahn-Hilliard equation \eqref{eq20170504_1} --
\eqref{eq20170504_3}, then we have 
\begin{align}
& \quad ~ 
\epsilon^2|u - P_hu|_{2,2,h} + \epsilon|u - P_h u|_{1,2,h} + \|u -
P_hu\|_{0,2} \label{eq20180817_7}\\ 
& \leq Ch \big( (\epsilon+h)^2|u|_{3,2} + |u|_{1,2} + \epsilon
    h\|u_t\|_{0,2} \big), \notag \\
& \quad ~ 
\epsilon^2|u_t - (P_hu)_t|_{2,2,h} + \epsilon|u_t - (P_h u)_t|_{1,2,h}
+ \|u_t - (P_hu)_t\|_{0,2} \label{eq20180817_8} \\
& \leq Ch \Big\{ (\epsilon+h)^2|u_t|_{3,2} + |u_t|_{1,2} +
\epsilon h \|u_{tt}\|_{0,2} + \|u_t\nabla u\|_{0,2} \notag \\
&\quad~+ \epsilon^{-1}|\ln
h|^{1/2}\|u_t\|_{0,2} ((\epsilon+h)^2|u|_{3,2} + |u|_{1,2} +
\epsilon h\|u_t\|_{0,2})
\Big\}. \notag
\end{align}
\end{theorem}
\begin{proof}
Taking $v = u$ and $v_h = P_hu$ in Lemma \ref{lem20180817_3}, and
noticing that  
$$ 
F_h(\psi) = \tilde{F}_h(\psi) = (\epsilon^4 \Delta^2u -
\epsilon^2\Delta f(u) + \alpha_0 u, \psi) = (\epsilon^3 u_t +
\alpha_0 u, \psi),
$$ 
we obtain the bound \eqref{eq20180817_7} from \eqref{eq20171006_1} and
\eqref{eq20180817_6}.

Taking $v = u_t$ and $v_h = (P_hu)_t$, we have 
$$ 
\begin{aligned}
F_h(\psi) &= (\epsilon^4 \Delta^2 u_t - \epsilon^2 \Delta f(u)_t +
\alpha_0 u_t, \psi) - (\epsilon^2 f''(u)u_t \nabla u, \nabla \psi)_h, \\
\widetilde{F}_h(\psi) &= (\epsilon^4 \Delta^2 u_t - \epsilon^2 \Delta
    f(u)_t + \alpha_0 u_t, \psi) - (\epsilon^2 f''(u)u_t \nabla P_hu,
    \nabla \psi)_h.
\end{aligned}
$$ 
Then we get 
$$ 
\begin{aligned}
& \quad ~F_h(\E_h \chi) - \widetilde{F}(\chi) + \alpha_0(u_t, \chi -
\E_h \chi) \\
&= (\epsilon^4 \Delta^2 u_t - \epsilon^2 \Delta f(u)_t, \E_h \chi -
    \chi) \\
& \quad ~- (\epsilon^2 f''(u)u_t\nabla u, \nabla \E_h\chi - \nabla\chi) -
(\epsilon^2 f''(u)u_t\nabla(u - P_hu), \nabla \chi)\\
&\leq \epsilon^3h^2\|u_{tt}\|_{0,2}|\chi|_{2,2,h} + C\epsilon^2 h
\|u_t\nabla u\|_{0,2} |\chi|_{2,2,h} + C\epsilon^2\|u_t\|_{0,2}\|\nabla
\chi\|_{0,\infty} |u - P_hu|_{1,2,h} \\
&\leq Ch\Big\{
\epsilon h\|u_{tt}\|_{0,2} + \|u_t\nabla u\|_{0,2} \\ 
&\quad~+ \epsilon^{-1}|\ln
h|^{1/2}\|u_t\|_{0,2} ((\epsilon+h)^2|u|_{3,2} + |u|_{1,2} +
\epsilon h\|u_t\|_{0,2})
\Big\}\triplenorm{\chi}_{2,2,h},
\end{aligned}
$$
where we use the discrete Sobolev inequality and the fact that $\nabla
\chi$ belongs to the Crouzeix-Raviar finite element space
\cite{brenner2015forty}. This implies the bound \eqref{eq20180817_8}.
\end{proof}

Combining with the a priori estimates of the bounds given in Section
\ref{sec2}, we have the following theorem. 
\begin{theorem} \label{thm20180818_1}
Assume $h \leq C\epsilon$, then there hold
\begin{align}
& \epsilon^4|u - P_hu|_{2,2,h}^2 + \epsilon^2|u - P_h
u|_{1,2,h}^2 + \|u - P_hu\|_{0,2}^2 
\leq Ch^2\rho_4(\epsilon), \label{eq20180818_1} \\ 
& \quad ~ \int_0^T \epsilon^4|u_t - (P_hu)_t|_{2,2,h}^2 +
\epsilon^2|u_t - (P_hu)_t|_{1,2,h}^2 + \|u_t - (P_hu)_t\|_{0,2}^2 \,\rd
s
\label{eq20180818_2} \\
&\leq Ch^2\epsilon^4\rho_3(\epsilon) + Ch^2|\ln h| \rho_5(\epsilon),
\notag
\end{align}
where 
$$ 
\begin{aligned}
\rho_4(\epsilon) &:=
\epsilon^{-\max\{2\sigma_1+\frac{13}{2},2\sigma_3+\frac{7}{2},2\sigma_2+4,2\sigma_4\}
+4}, \\
\rho_5(\epsilon) &:=
\epsilon^{-2\max\{2\sigma_1+\frac{13}{2},2\sigma_3+\frac{7}{2},2\sigma_2+4,2\sigma_4\}
+2}.
\end{aligned}
$$ 
\end{theorem}
\begin{proof}
Using \eqref{eq2.5}, \eqref{eq2.13_add} and \eqref{eq2.15_add}, we have 
\begin{align}
& \quad ~(\epsilon+h)^4 |u|_{3,2}^2 + |u|_{1,2}^2 +
\epsilon^2 h^2 \|u_t\|_{0,2}^2 \label{eq20180818_3} \\
& \leq C\epsilon^{-\max\{2\sigma_1+5, 2\sigma_3+2\}+4} +
C\epsilon^{-2\sigma_1 - 1} + 
C\epsilon^{-\max\{ 
2\sigma_1+\frac{13}{2},2\sigma_3+\frac{7}{2},2\sigma_2+4,2\sigma_4
\}+4} \notag \\
& \leq C\rho_4(\epsilon), \notag
\end{align}
which implies the bound \eqref{eq20180818_1} by \eqref{eq20180817_7}.

Using \eqref{eq2.15_add}, \eqref{eq20180606_2}, \eqref{eq2.18} and
\eqref{eq20180801_3}, we obtain 
$$ 
\begin{aligned}
&\quad~ \int_0^T (\epsilon+h)^4|u_t|_{3,2}^2 +
|u_t|_{1,2}^2 +
\epsilon^2 h^2 \|u_{tt}\|_{0,2}^2 + \|u_t\nabla u\|_{0,2}^2\,\rd s \\
& \leq C\int_0^T \epsilon^4|u_t|_{3,2}^2
+ |u_t|_{1,2}^2
+ \epsilon^4\|u_{tt}\|_{0,2}^2
+ \|u_t\|_{0,2}^2 \|\nabla u\|_{0,\infty}^2 \,\rd s \\
& \leq C\epsilon^3\rho_0(\epsilon) + C\rho_0(\epsilon) +
C\epsilon^4\rho_3(\epsilon) \\ 
&~+ C\epsilon^{-\max\{\sigma_1+\frac52,
  \sigma_3+1\} - \max\{2\sigma_1+\frac{13}{2}, 2\sigma_3+\frac72,
  2\sigma_2+4, 2\sigma_4 \}} \\
& \leq C\epsilon^4\rho_3(\epsilon).
\end{aligned}
$$ 
Further, using \eqref{eq2.15_add} and \eqref{eq20180818_3}, we obtain 
$$ 
\begin{aligned}
\quad ~ \int_0^T \epsilon^{-2} \|u_t\|_{0,2}^2
((\epsilon+h)^2|u|_{3,2} + |u|_{1,2} + \epsilon
 h\|u_t\|_{0,2})^2\,\rd s \leq C \rho_5(\epsilon). 
\end{aligned}
$$ 
This implies the bound \eqref{eq20180818_2}.
\end{proof}

\begin{corollary}\label{cor20180818_1}
Under the condition that 
\begin{equation} \label{mesh_cond}
h \leq C\epsilon^2 \rho_4^{-\frac12}(\epsilon), \quad 
h \leq C \rho_3^{-\frac12}(\epsilon), \quad 
h|\ln h|^{\frac12} \leq C \epsilon^2 \rho_5^{-\frac12}(\epsilon),
\end{equation}
there hold 
\begin{align} \label{eq20180819_1}  
|P_hu|_{j,2,h}^2 &\leq C(1+|u|_{j,2,h}^2) \quad j=0,1,2, \\
\int_0^T |P_hu|_{j,2,h}^2 \,\rd s&\leq C(1+\int_0^T |u|_{j,2,h}^2) \quad
j=0,1,2, \notag \\
\|P_h u\|_{0,\infty} &\leq C. \notag
\end{align}
\end{corollary}
\begin{proof}
By the Sobolev embedding and \eqref{eq20180818_1}, we have 
$$ 
\|P_h u\|_{0,\infty} \leq \|u\|_{0,\infty} + \|u - P_hu\|_{2,2,h} \leq 
C + Ch\epsilon^{-2}\rho_4^{1/2}(\epsilon) \leq C.
$$ 
The first two bounds are the direct consequences of Theorem
\ref{thm20180818_1}. 
\end{proof}

\section{Error Estimates} \label{sec4}

In this section, first we derive the piecewise
${L^{\infty}(L^2)}$ and ${L^2(H^2)}$ error bounds which depend on $\frac{1}{\epsilon}$ polynomially based on the generalized coercivity result in Theorem \ref{thm3.7_add}, and piecewise ${L^{\infty}(H^{-1})}$ and
${L^2(H^1)}$ error bounds. Then we prove the piecewise
${L^{\infty}(H^2)}$ error bound based on the piecewise
${L^{\infty}(L^2)}$ and ${L^2(H^2)}$ error bounds.
Finally, the ${L^\infty(L^\infty)}$ error bound is
established.

Decompose the error
\begin{align}
u-u_h^n=(u-P_hu)+(P_hu-u_h^n):=\rho^n+\theta^n.
\end{align}

The following two lemmas will be used in this section. 
\begin{lemma}[Summation by parts] \label{lem20180515_1} 
Suppose $\{a_n\}_{n=0}^\ell$ and $\{b_n\}_{n=0}^\ell$ are two
sequences, then
\begin{equation*}
\sum_{n=1}^\ell(a^n-a^{n-1},b^n) =
(a^\ell,b^\ell)-(a^0,b^0)-\sum_{n=1}^\ell(a^{n-1},b^n-b^{n-1}).
\end{equation*}
\end{lemma}

\begin{lemma}\label{lem20180409_1} 
Suppose $u(t_n)$ to be the solution of
\eqref{eq20170504_1}--\eqref{eq20170504_5}, and $u_h^n$ to be the
solution of \eqref{eq20170504_11}--\eqref{eq20170504_12}, then
\begin{align*}
\rho^n \in \mathring{S}^h_E,\quad \theta^n \in \mathring{S}^h_E.
\end{align*}
\end{lemma}
\begin{proof}
Testing \eqref{eq20170504_1} with constant $1$, and then taking the
integration over $(0,t)$, we can obtain for any $t\ge0$,
\begin{align*}
\int_{\Omega}u(t)dx=\int_{\Omega}u(0)dx.
\end{align*}
Then choosing $v=u(t), w=1$ in \eqref{eq20170504_14}, we have for any
$t\ge0$,
\begin{align*}
\int_{\Omega}P_hu(t)~\rd x=\int_{\Omega}u(t)~\rd x.
\end{align*}
Choosing $v_h=1$ in \eqref{eq20170504_11}, then
\begin{align*}
\int_{\Omega}u_h^n\,\rd x=\int_{\Omega}u_h^{n-1}\,\rd x =
\cdots=\int_{\Omega}u_h^{0}\,\rd x.
\end{align*}
Therefore, if choosing $u_h^{0}=P_hu(0)$, then
\begin{align*}
\int_{\Omega}u_h^n\,\rd x &= \int_{\Omega}u_h^{0}\,\rd x 
=\int_{\Omega}P_hu(0)\,\rd x\\
&=\int_{\Omega}u(0)\,\rd x =\int_{\Omega}u(t_n)\,\rd x
=\int_{\Omega}P_hu(t_n)\,\rd x.
\end{align*}
Hence, $P_hu(t_n)-u_h^n \in \mathring{S}^h_E$.
\end{proof}

\subsection{Generalized coercivity result, piecewise $L^\infty(H^{-1})$ and
  $L^2(H^1)$ error estimates}
We first cite the generalized coercivity result, piecewise
$L^{\infty}(H^{-1})$ and $L^2(H^1)$ error estimates established in
\cite{li2017error}.  

\begin{theorem}[Generalized coercivity] \label{thm3.7_add}
Suppose there exists a positive number $\gamma_3>0$ such that the
solution $u$ of problem \eqref{eq20170504_1}--\eqref{eq20170504_5} and
elliptic operator $P_h$ satisfy
\begin{equation} \label{eq20180819_4} 
\|u-P_h u\|_{L^{\infty}((0,T);L^{\infty})} \leq C_1 h
\epsilon^{-\gamma_3}.
\end{equation}
Then there exists an $\eps$-independent and $h$-independent constant $C>0$
such that for $\eps\in(0,\eps_0)$, a.e. $t\in [0,T]$, and for any $\psi\in
\mathring{S}_E^h$,
\begin{equation*}
(\epsilon-\epsilon^4)(\nabla\psi,\nabla\psi)_h+
\frac{1}{\epsilon}(f'(P_hu(t))\psi,\psi)_h\geq
-C\|\nabla\Delta^{-1}\psi\|_{L^2}^2-C\epsilon^{-2\gamma_2-4}h^4,
\end{equation*}
provided that $h$ satisfies the constraint
\begin{align}\label{eq3.24b}
h &\leq (C_1C_2)^{-1}\eps^{\gamma_3+3},  
\end{align} 
where $\gamma_2 = 2\gamma_1 + \sigma_1 + 6$ and $C_2$ is determined by  
$$
C_2:=\max_{|\xi|\le \|u\|_{L^\infty((0,T); L^\infty)}}|f{''}(\xi)|.
$$
\end{theorem}
\begin{remark} Thanks to the Sobolev embedding theorem and
\eqref{eq20180818_1}, we have 
\begin{equation} \label{eq20180819_5} 
\|u - P_hu\|_{0,\infty} \leq \|u - P_hu\|_{2,2,h} \leq
Ch \epsilon^{-2}\rho_4^{\frac12}(\epsilon), 
\end{equation} 
which gives the explicit formulation of $\gamma_3$ in
\eqref{eq20180819_4}.
\end{remark}

\begin{theorem}[Piecewise $L^\infty(H^{-1})$ and $L^2(H^1)$ error
estimates]\label{thm20171007_1}
Assume $u$ is the solution of
\eqref{eq20170504_1}--\eqref{eq20170504_5}, $u_h^n$ is the numerical
solution of scheme \eqref{eq20170504_11}--\eqref{eq20170504_12}.
Under the mesh constraints in Theorem 3.15 in \cite{li2017error}, we
have the following error estimate
\begin{align*}
&\frac{1}{4}\|\nabla
\widetilde\Delta_h^{-1}\theta^{\ell}\|_{0,2,h}^2 +
\frac{k^2}{4}\sum_{n=1}^\ell\|\nabla
\widetilde\Delta_h^{-1}d_t\theta^n\|_{0,2,h}^2 +
\frac{\epsilon^4k}{16}\sum_{n=1}^\ell(\nabla\theta^n,\nabla\theta^n)_h\\
& \qquad + \frac{k}{\epsilon}\sum_{n=1}^\ell\|\theta^n\|_{0,4,h}^4\leq
C(\tilde{\rho}_0(\epsilon) |\ln h| h^2 +
\tilde{\rho}_1(\epsilon) k^2),
\end{align*}
where $\tilde{\rho}_0(\epsilon)$ and $\tilde{\rho}_1(\epsilon)$ are
polynomial $\frac1\epsilon$-dependent functions and
$\widetilde{\Delta}_h^{-1}$ is a discrete inverse Laplace operator
defined in \cite{li2017error}. 
\end{theorem}

\subsection{$L^\infty(L^2)$ and piecewise $L^2(H^2)$ error estimates}
Based on Theorem \ref{thm20171007_1}, the $L^\infty(L^2)$ and piecewise $L^2(H^2)$ error estimates which
depend on $\frac{1}{\epsilon}$ polynomially, instead of
exponentially, are derived below. Notice that the Theorem
\ref{thm20171007_1} is used to circumvent the use of interpolation
of $\|\cdot\|_{1,2,h}$ between $\|\cdot\|_{0,2,h}$ and
$\|\cdot\|_{2,2,h}$, by which only the exponential dependence can be
derived.

\begin{theorem}\label{thm20180611_add_1}
Assume $u$ is the solution of
\eqref{eq20170504_1}--\eqref{eq20170504_5}, $u_h^n$ is the numerical
solution of scheme \eqref{eq20170504_11}--\eqref{eq20170504_12}.
Under the mesh constraints in Theorem 3.15 in \cite{li2017error} and \eqref{mesh_cond},
the following $L^\infty(L^2)$ and piecewise $L^2(H^2)$ error estimates hold
\begin{align} \label{eq20180820_1}
&~\quad
\|\theta^{\ell}\|_{0,2,\Omega}^2+k\sum_{n=1}^\ell\|d_t\theta^{n}\|_{0,2,\Omega}^2
+ \epsilon k\sum_{n=1}^{\ell}a_h(\theta^n,\theta^n)\\
& \leq C\tilde{\rho}_2(\epsilon) |\ln h|^2 h^2 +
C\tilde{\rho}_3(\epsilon) |\ln h| k^2,
\notag
\end{align}
where 
$$ 
\begin{aligned}
\tilde{\rho}_2(\epsilon) &:= \epsilon^{4}\rho_3(\epsilon) +
\epsilon^{-2\sigma_1-6}\rho_4(\epsilon) + \rho_5(\epsilon) +
\epsilon^{-5}\tilde{\rho}_0(\epsilon) + 
\epsilon^{-2\gamma_1 - 2\gamma_2 - 2}\tilde{\rho}_0(\epsilon), \\
\tilde{\rho}_3(\epsilon) &:= \rho_3(\epsilon) +
\epsilon^{-5}\tilde{\rho}_1(\epsilon) + \epsilon^{-2\gamma_1 -
  2\gamma_2 - 2}\tilde{\rho}_1(\epsilon).
\end{aligned}
$$ 
\end{theorem}
\begin{proof}
It follows from \eqref{eq20170504_11}, \eqref{eq20170504_14}, and \eqref{eq20170504_15} that for any $v_h\in S^h_E$,
\begin{align}\label{eq20180209_2}
&(d_t\theta^n,v_h)+\epsilon a_h(\theta^n,v_h)\\
=&~ [(d_tP_hu,v_h)+\epsilon a_h(P_hu,v_h)]-[(d_tu_h^n,v_h)+\epsilon a_h(u_h^n,v_h)]\notag\\
=& -(d_t\rho^n,v_h)+(u_t+\epsilon\Delta^2u-\frac{1}{\epsilon}\Delta f(u)+\alpha u,v_h)+(R^n(u_{tt}),v_h)\notag\\
&-\frac{1}{\epsilon}(f'(u)\nabla P_hu,\nabla v_h)_h-\alpha(P_hu,v_h)+\frac{1}{\epsilon}(\nabla f(u_h^{n}),\nabla v_h)_h\notag\\
=&~(-d_t\rho^n+\alpha\rho^n,v_h)-\frac{1}{\epsilon}(f'(u)\nabla P_hu-\nabla f(u_h^n),\nabla v_h)_h\notag\\
&+(R^n(u_{tt}),v_h), \notag
\end{align}
where the remainder 
\begin{equation} \label{eq20180801_2} 
R^n(u_{tt}):= \frac{u(t_n) - u(t_{n-1})}{k} - u_t(t_n) =
-\frac{1}{k}\int^{t_n}_{t_{n-1}}(s-t_{n-1})u_{tt}(s)\,\rd s.
\end{equation}
Choosing $v_h=\theta^n$, taking summation over $n$ from $1$ to $\ell$,
multiplying $k$ on both sides of \eqref{eq20180209_2}, we have  
\begin{align} \label{eq20180801_1}
& ~\quad \frac{1}{2}\|\theta^\ell\|_{0,2}^2 +
\frac{k}{2}\sum_{n=1}^\ell \|d_t \theta^n\|_{0,2}^2 + \epsilon k
\sum_{n=1}^\ell a_h(\theta^n, \theta^n) \\
&= k\sum_{n=1}^\ell (-d_t\rho^n+\alpha\rho^n,\theta^n) -
\frac{k}{\epsilon}\sum_{n=1}^\ell (f'(u)\nabla
P_hu-\nabla f(u_h^n),\nabla \theta^n)_h\notag\\
&~\quad+k\sum_{n=1}^\ell(R^n(u_{tt}),\theta^n):= I_1 + I_2 + I_3. \notag
\end{align}
     
\noindent\underline{Estimate of $I_1$:} The first term on the right hand side
of \eqref{eq20180209_2} can be bounded by
\begin{align}\label{eq20180211_5}
I_1 &= k\sum_{n=1}^\ell (-d_t\rho^n+\alpha\rho^n,\theta^n)\\
&\le Ck \sum_{n=1}^\ell \|d_t\rho^n\|_{0,2}^2 + Ck\sum_{n=1}^\ell 
\alpha^2 \|\rho^n\|_{0,2}^2 + Ck \sum_{n=1}^\ell
\|\theta^n\|_{0,2}^2 \notag\\
&\le C (\epsilon^4\rho_3(\epsilon) + \epsilon^{-6}\rho_4(\epsilon)) h^2
+ C\rho_5(\epsilon)|\ln h|h^2 +Ck\sum_{n=1}^\ell \|\theta^n\|_{0,2}^2,\notag
\end{align}
where by \eqref{eq20180818_1} and \eqref{eq20180818_2}
\begin{align}
k \sum_{n=1}^\ell \|d_t \rho^n\|_{0,2}^2 & = \frac{1}{k}
\sum_{n=1}^\ell \|\int_{t_{n-1}}^{t_n} \rho_t \,\rd s\|_{0,2}^2 
\le \sum_{n=1}^\ell \int_{t_{n-1}}^{t_n} \|\rho_t \|_{0,2}^2 \,\rd s
\label{eq20180818_add1}\\
& \leq \int_0^T \|\rho_t\|_{0,2}^2 \,\rd s \leq
C\epsilon^4\rho_3(\epsilon)h^2 + C \rho_5(\epsilon)|\ln h|h^2, \notag \\  
k \sum_{n=1}^\ell \alpha^2 \|\rho^n\|_{0,2}^2 & \le C \epsilon^{-6}
\supn \|\rho^n\|_{0,2}^2 \leq C\epsilon^{-6}\rho_4(\epsilon)h^2.
\label{eq20180818_add2}
\end{align}

\noindent\underline{Estimate of $I_2$:} The second term on the right
hand side of \eqref{eq20180801_1} can be written as
\begin{align}\label{eq20180209_3}
&-\frac{k}{\epsilon} \sum_{n=1}^\ell (f'(u)\nabla P_hu-\nabla
    f(u_h^n),\nabla\theta^n)_h\\
=&-\frac{k}{\epsilon}\sum_{n=1}^\ell (f'(u)\nabla P_hu-f'(P_hu)\nabla
    P_hu,\nabla\theta^n)_h \notag \\ 
&- \frac{k}{\epsilon}\sum_{n=1}^\ell(\nabla f(P_hu)-
      f'(P_hu)\nabla u_h^n,\nabla\theta^n)_h \notag\\
&-\frac{k}{\epsilon}\sum_{n=1}^\ell (f'(P_hu)\nabla u_h^n-\nabla
    f(u_h^n),\nabla\theta^n)_h := J_1 + J_2 + J_3. \notag 
\end{align}

By \eqref{eq2.5}, \eqref{eq20180818_1} and mesh condition \eqref{mesh_cond},
we have 
$$ 
\|\nabla P_hu\|_{0,2}^2 \leq \|\nabla u\|_{0,2}^2 + C \leq \epsilon^{-2\sigma_1 - 1}.
$$ 
Then, using \eqref{eq20180819_5} and the piecewise $L^2(H^1)$ error estimate 
given in Theorem \ref{thm20171007_1}, the first term on the right-hand side
of \eqref{eq20180209_3} can be bounded below
\begin{align}
J_1 &= -\frac{3k}{\epsilon}\sum_{n=1}^\ell (\rho^n(u+P_hu)\nabla P_hu,
\nabla\theta^n)_h \label{eq20180211_5_add}\\
& \leq \frac{Ck}{\epsilon}\sum_{n=1}^\ell \|u+P_hu\|_{0,\infty}^2
\|\rho^n\|_{0,\infty}^2 \|\nabla P_h u\|_{0,2}^2 +
\frac{Ck}{\epsilon}\sum_{n=1}^\ell (\nabla \theta^n, \nabla
\theta^n)_h \notag \\
& \leq C \epsilon^{-2\sigma_1-6}\rho_4(\epsilon) h^2 +  
C\epsilon^{-5}\tilde{\rho}_0(\epsilon){|\ln h|}h^2 +
C\epsilon^{-5}\tilde{\rho}_1(\epsilon)k^2. \notag
\end{align}

Again, thanks to the piecewise $L^2(H^1)$ error estimate given in Theorem
\ref{thm20171007_1}, the second term on the right-hand side of
\eqref{eq20180209_3} can be written as
\begin{align}\label{eq20180209_4}
J_2 & = -\frac{k}{\epsilon}\sum_{n=1}^\ell (f'(P_hu)\nabla
    \theta^n,\nabla\theta^n)_h \leq \frac{Ck}{\epsilon}\sum_{n=1}^\ell
(\nabla \theta^n, \nabla \theta^n)_h \\ 
&\leq C\epsilon^{-5}\tilde{\rho}_0(\epsilon){|\ln h|}h^2 +
C\epsilon^{-5}\tilde{\rho}_1(\epsilon)k^2. \notag
\end{align}

By the discrete Sobolev inequality and Theorem 3.14 in
\cite{li2017error}, we have for any $n$,
\begin{align}\label{eq20180212_1}
\|u_h^n\|_{1,\infty,h} \le C|\ln h|^{\frac12}\|u_h^n\|_{2,2,h} 
\le C\epsilon^{-\gamma_2}|\ln h|^{\frac12}.
\end{align}
Then, the third term on the right-hand side of \eqref{eq20180209_3}
can be bounded by
\begin{align}\label{eq20180213_1}
J_3 & = -\frac{3k}{\epsilon} \sum_{n=1}^\ell (\theta^n(P_hu + u_h)
    \nabla u_h^n,\nabla\theta^n) \\ 
& \leq Ck \sum_{n=1}^\ell \|\theta^n\|_{0,2}^2 + \frac{Ck}{\epsilon^2}
\sum_{n=1}^\ell \|P_hu + u_h^n\|_{0,\infty}^2 \|u_h^n\|_{1,\infty,h}^2
\|\nabla \theta^n\|_{0,2}^2 \notag \\
& \leq Ck \sum_{n=1}^\ell \|\theta^n\|_{0,2}^2 + C\epsilon^{-2\gamma_1
- 2\gamma_2 - 2} |\ln h| k\sum_{n=1}^\ell \|\nabla \theta^n\|_{0,2}^2
\notag \\
& \leq Ck \sum_{n=1}^\ell \|\theta^n\|_{0,2}^2 + C\epsilon^{-2\gamma_1
- 2\gamma_2 - 2} (\tilde{\rho}_0(\epsilon){ |\ln h|^2}h^2 +
\tilde{\rho}_1(\epsilon)|\ln h|k^2 ).\notag
\end{align}

\noindent\underline{Estimate of $I_3$:} The third term on the right
hand side of \eqref{eq20180209_2} can be bounded by
\begin{align}\label{eq20180517_12}
I_3 = k\sum_{n=1}^\ell (R^n(u_{tt}),\theta^n)
& \le Ck \sum_{n=1}^\ell \|R^n(u_{tt})\|_{0,2}^2 + Ck\sum_{n=1}^\ell
\|\theta^n\|_{0,2}^2 \\ 
& \le C\rho_3(\epsilon)k^2 + Ck\sum_{n=1}^\ell \|\theta^n\|_{0,2}^2
\notag,
\end{align}
where by \eqref{eq20180606_2} and \eqref{eq20180801_2},
\begin{align}
k\sum_{n=1}^{\ell}\|R^n(u_{tt})\|_{0,2}^2 
&\leq \frac{1}{k}\sum_{n=1}^{\ell} \Bigl(\int^{t_n}_{t_{n-1}}(s-t_{n-1})^2\,\rd s\Bigr)
\Bigl(\int^{t_n}_{t_{n-1}}\|u_{tt}(s)\|_{0,2}^2\,\rd s\Bigr)\label{eq20180606_11}\\
&\leq C\rho_3(\eps)k^2.\notag
\end{align}

\noindent{\underline{$L^\infty(L^2)$  and piecewise $L^2(H^2)$ error estimates:}} Taking
\eqref{eq20180211_5}, \eqref{eq20180211_5_add}, \eqref{eq20180209_4},
\eqref{eq20180213_1}, \eqref{eq20180213_1} into \eqref{eq20180801_1},
we have 
\begin{align}\label{eq20180209_4_add}
& \quad ~ \frac{1}{2} \|\theta^{\ell}\|_{0,2}^2 + \frac{k}{2}
\sum_{n=1}^\ell\|d_t\theta^{n}\|_{0,2}^2 + \epsilon
k\sum_{n=1}^{\ell}a_h(\theta^n,\theta^n) \\
&\le Ck\sum_{n=1}^{\ell} \|\theta^n\|_{0,2}^2 \notag \\ 
&~~~ + C(\epsilon^{4}\rho_3(\epsilon) +
    \epsilon^{-2\sigma_1-6}\rho_4(\epsilon) ) h^2 
\notag\\
&~~~ + C(\rho_5(\epsilon) +
\epsilon^{-5}\tilde{\rho}_0(\epsilon))|\ln h| h^2
+ \epsilon^{-2\gamma_1 - 2\gamma_2 -
    2}\tilde{\rho}_0(\epsilon)|\ln h|^2h^2 \notag \\ 
&~~~ 
+ C(\rho_3(\epsilon) + \epsilon^{-5}\tilde{\rho}_1(\epsilon))k^2
+ C\epsilon^{-2\gamma_1 - 2\gamma_2 -
  2}\tilde{\rho}_1(\epsilon)|\ln h|k^2. \notag
\end{align}
The desired result \eqref{eq20180820_1} is therefore obtained by the Gronwall's
inequality.
\end{proof}

\subsection{Piecewise $L^\infty(H^2)$ and $L^\infty(L^\infty)$ error estimates}
In this subsection, we give the $\|\theta^\ell\|_{2,2,h}^2$ estimate
by taking the summation by parts in time and integration by parts in
space, and using the special properties of the Morley element. The
$\|\theta^\ell\|_{2,2,h}^2$ estimate below is ``almost'' optimal with
respect to time and space.

\begin{theorem}\label{thm20180214_4}
Assume $u$ is the solution of
\eqref{eq20170504_1}--\eqref{eq20170504_5}, $u_h^n$ is the numerical
solution of scheme \eqref{eq20170504_11}--\eqref{eq20170504_12}.
Under the mesh constraints in Theorem 3.15 in \cite{li2017error} and
\eqref{mesh_cond}, the following piecewise $L^\infty(H^2)$ error
estimate holds
\begin{align}
&\quad ~ k\sum_{n=1}^{\ell}\|d_t\theta^n\|_{L^2}^2 
+\epsilon k^2 \sum_{n=1}^{\ell}a_h(d_t\theta^n, d_t\theta^n)
  +\epsilon\|\theta^\ell\|_{2,2,h}^2 \label{eq20180821_1}\\
&\le C\tilde{\rho}_4(\epsilon) |\ln h|^2 h^2 
+ C\tilde{\rho}_5(\epsilon)|\ln h| k^2, \notag
\end{align}
\end{theorem}
where
\begin{align*}
\tilde{\rho}_4(\epsilon)& = \epsilon^{-2\sigma_1 - 1}\rho_3(\epsilon)
  + \epsilon^{-4}\rho_0(\epsilon)\rho_4(\epsilon) +
  \epsilon^{-2\sigma_1 - 5}\rho_5(\epsilon) \\
&~~~+ \Big(\epsilon^{-4\gamma_1-3} + \epsilon^{-4\gamma_2-2} 
+ \epsilon^{-\max\{2\sigma_1+5,2\sigma_3+2\}-2} \notag \\
&\qquad~ + \epsilon^{2\gamma_1-\max\{
 2\sigma_1 + \frac{13}{2}, 2\sigma_{3} + \frac72, 2\sigma_2+4,
 2\sigma_4 \} - 1} \Big) \tilde{\rho}_2(\epsilon),
\\ 
\tilde{\rho}_5(\epsilon)& = \Big(\epsilon^{-4\gamma_1-3} + \epsilon^{-4\gamma_2-2} 
+ \epsilon^{-\max\{2\sigma_1+5,2\sigma_3+2\}-2} \notag \\
&\qquad~ + \epsilon^{2\gamma_1-\max\{
 2\sigma_1 + \frac{13}{2}, 2\sigma_{3} + \frac72, 2\sigma_2+4,
 2\sigma_4 \} - 1} \Big) \tilde{\rho}_3(\epsilon).
\end{align*}
\begin{proof}
Choosing $v_h = \theta^n - \theta^{n-1} = kd_t \theta^n$ in \eqref{eq20180209_2},
taking summation over $n$ from $1$ to $\ell $, we get 
\begin{align}\label{eq20180214_2}
&\quad ~ k \sum_{n=1}^\ell \|d_t\theta^n\|_{L^2}^2 +
\frac{\epsilon}{2} a_h(\theta^\ell, \theta^\ell) + 
\frac{\epsilon k^2}{2} \sum_{n=1}^\ell a_h(d_t\theta^n,d_t\theta^n)\\
&= k\sum_{n=1}^\ell (-d_t\rho^n+\alpha\rho^n,d_t\theta^n) 
-\frac{k}{\epsilon}\sum_{n=1}^\ell (f'(u)\nabla
P_hu-\nabla f(u_h^n),\nabla(d_t\theta^n))_h\notag\\
&\quad + k \sum_{n=1}^\ell (R^n(u_{tt}),d_t\theta^n) := I_1 + I_2 +
I_3. \notag
\end{align}
Here we use the fact that 
\begin{align*}
\epsilon a_h(\theta^n,\theta^n-\theta^{n-1}) =
\frac{\epsilon k^2}{2}a_h(d_t \theta^n, d_t \theta^n) +\frac{\epsilon}{2}a_h(\theta^n,\theta^n)-\frac{\epsilon}{2}a_h(\theta^{n-1},\theta^{n-1}).\notag
\end{align*}

\noindent\underline{Estimates of $I_1$ and $I_3$}: Similar to
\eqref{eq20180211_5}, using \eqref{eq20180818_add1} and
\eqref{eq20180818_add2}, we have 
\begin{align}
I_1 &\le Ck \sum_{n=1}^\ell \|d_t\rho^n\|_{L^2}^2 + Ck\sum_{n=1}^\ell 
\alpha^2 \|\rho^n\|_{L^2}^2 + \frac{k}{8} \sum_{n=1}^\ell
\|d_t \theta^n\|_{L^2}^2 \label{eq20180820_2}\\
&\le C (\epsilon^4\rho_3(\epsilon) + \epsilon^{-6}\rho_4(\epsilon)) h^2
+ C\rho_5(\epsilon)|\ln h|h^2 + \frac{k}{8}\sum_{n=1}^\ell
\|d_t \theta^n\|_{0,2,h}^2.\notag
\end{align}
From \eqref{eq20180517_12} and \eqref{eq20180606_11}, we also obtain
the estimate of $I_3$ below
\begin{align}\label{eq20180820_3}
I_3 = k\sum_{n=1}^\ell (R^n(u_{tt}),d_t\theta^n)
& \le Ck \sum_{n=1}^\ell \|R^n(u_{tt})\|_{L^2}^2 + \frac{k}{8}\sum_{n=1}^\ell
\|d_t\theta^n\|_{0,2}^2 \\ 
& \le C\rho_3(\epsilon)k^2 + \frac{k}{8} \sum_{n=1}^\ell
\|d_t \theta^n\|_{0,2}^2 \notag.
\end{align}

\noindent\underline{Estimate of $I_2$}: Next we bound the more
complicated term $I_2$. Using integration by parts, we have 

\begin{align} \label{eq20180822_1}
I_2 &= -\frac{k}{\epsilon}\sum_{n=1}^\ell (f'(u)\nabla P_hu - \nabla
    f(P_hu), d_t\nabla \theta^n)_h - \frac{k}{\epsilon}\sum_{n=1}^\ell (\nabla( f(P_hu) - f(u_h^n)),
      d_t\nabla \theta^n)_h \\ 
&= -\frac{k}{\epsilon}\sum_{n=1}^\ell (f'(u)\nabla P_hu - \nabla
    f(P_hu), d_t\nabla \theta^n)_h + \frac{k}{\epsilon}\sum_{n=1}^\ell (f(P_hu) - f(u_h^n), d_t
    \Delta \theta^n)_h \notag \\
&~~~ - \frac{k}{\epsilon}\sum_{n=1}^\ell \sum_{E \in \mathcal{E}_h}
(\{f(P_hu) - f(u_h^n)\}, d_t \llbracket\nabla \theta^n\rrbracket )_E
\notag \\
&~~~ - \frac{k}{\epsilon}\sum_{n=1}^\ell \sum_{E\in \mathcal{E}_h}
(\llbracket f(P_hu) - f(u_h^n)\rrbracket, \{\nabla d_t\theta^n\})_E
\notag := J_1 + J_2 + J_3 + J_4.
\end{align}
Here we adopt the standard DG notation and the DG identity, see
\cite[Equ. (3.3)]{arnold2002unified}.  Next we bound $J_1$ to $J_4$
respectively. 

\paragraph{$\bullet$ Estimate of $J_1$}
Using summation by parts in Lemma \ref{lem20180515_1}, we have
\begin{align} \label{eq20180822_5}
J_1 &= \frac{k}{\epsilon}\sum_{n=1}^\ell(d_t (\rho(u+P_hu)\nabla P_hu),
    \nabla\theta^{n-1})_h - \frac{1}{\epsilon}(\rho^\ell(u^\ell +
        P_hu^\ell)\nabla P_hu^\ell, \nabla \theta^\ell )_h.
\end{align}
Thanks to \eqref{eq2.5}, \eqref{eq2.15_add}, \eqref{eq2.18},
\eqref{eq20180818_1}, \eqref{eq20180818_2}, \eqref{eq20180819_1}, and
the piecewise $L^2(H^1)$ estimate in Theorem \ref{thm20171007_1}, the
first term on the right hand side of \eqref{eq20180822_5} can be
bounded by 
\begin{align} \label{eq20180822_2}
& \quad ~ \frac{k}{\epsilon}\sum_{n=1}^\ell(d_t (\rho(u+P_hu)\nabla P_hu),
\nabla\theta^{n-1})_h \\
& \leq \frac{1}{k} \sum_{n=1}^\ell \| \int_{t_{n-1}}^{t_n} (\rho(u+P_hu)\nabla
    P_hu)_t \,\rd s\|_{0,2}^2 + C\epsilon^{-2} k \sum_{n=1}^\ell
|\theta^{n-1}|_{1,2,h}^2 \notag \\
& \leq \esssupT\|\nabla P_h u\|_{0,2}^2
\int_0^T\|\rho_t\|_{0,\infty}^2 \,\rd s +
\esssupT\|\rho\|_{0,\infty}^2 \int_0^T \|\nabla(P_hu)_t\|_{0,2}^2
\,\rd s
\notag \\
& + \esssupT \|\rho\|_{0,\infty}^2 \|\nabla P_hu \|_{0,2}^2 \int_0^T
\|u_t + (P_hu)_t\|_{0,\infty}^2 \,\rd s + C\epsilon^{-2} k \sum_{n=1}^\ell
|\theta^{n-1}|_{1,2,h}^2
\notag \\
& \leq C\epsilon^{-2\sigma_1 - 1}(\rho_3(\epsilon) +
\epsilon^{-4}\rho_5(\epsilon)|\ln h|)h^2 + C
\epsilon^{-4}\rho_0(\epsilon)\rho_4(\epsilon)h^2 
\notag \\
&~~~ + C\epsilon^{-2\sigma_1 - 6 -\max\{2\sigma_1 +
  \frac{13}{2}, 2\sigma_3+ \frac72, 2\sigma_2 + 4,
  2\sigma_4\}}\rho_4(\epsilon) h^2 \notag \\ 
&~~~ + C\epsilon^{-6}\tilde{\rho}_0(\epsilon)|\ln h|h^2 +
C\epsilon^{-6}\tilde{\rho}_1(\epsilon)k^2.
\notag
\end{align}
Thanks to \eqref{eq2.5}, \eqref{eq20180818_1} and the $L^\infty(L^2)$
estimate in Theorem \ref{thm20180611_add_1}, the second term on the
right hand of \eqref{eq20180822_5} can be bounded by 
\begin{align} \label{eq20180822_3}
&\quad ~ - \frac{1}{\epsilon}(\rho^\ell(u^\ell + P_hu^\ell)\nabla
P_hu^\ell, \nabla \theta^\ell )_h \\
& \leq C \epsilon^{-2}\|\rho^l\|_{0,\infty}^2 |P_hu^l|_{1,2,h}^2 +
C\epsilon^{-1}\|\theta\|_{0,2}^2 + \frac{\epsilon}{8}a_h(\theta^l,
    \theta^l) \notag \\
& \leq \notag C\epsilon^{-2\sigma_1 - 7}\rho_4(\epsilon) h^2 +
C\epsilon^{-1}\tilde{\rho}_2(\epsilon)|\ln h|^2 h^2 + 
C\epsilon^{-1}\tilde{\rho}_3(\epsilon)|\ln h|k^2 + 
\frac{\epsilon}{8}a_h(\theta^l, \theta^l).
\end{align}
Combining \eqref{eq20180822_2} and \eqref{eq20180822_3}, simplifying
the coefficients according to the definition of $\rho_i(\epsilon)$ and
$\tilde{\rho}_i(\epsilon)$, we obtain the bound for $J_1$:
\begin{align} \label{eq20180822_4}
J_1 &\leq 
C(\epsilon^{-2\sigma_1 - 1}\rho_3(\epsilon) +
    \epsilon^{-4}\rho_0(\epsilon)\rho_4(\epsilon) +
    \epsilon^{-2\sigma_1-5}\rho_5(\epsilon) + 
    \epsilon^{-1}\tilde{\rho}_2(\epsilon) ) |\ln h|^2 h^2 \\ 
&~~~ + C\epsilon^{-1}\tilde{\rho}_3(\epsilon)|\ln h|k^2 +
\frac{\epsilon}{8}a_h(\theta^l, \theta^l). \notag
\end{align}

\paragraph{$\bullet$ Estimate of $J_2$}
Define $f(P_hu) - f(u_h^n) := M^n\theta^n$, where $M^n$ is given as
$$
M^n := (P_hu(t_n))^2 + P_hu(t_n)u_h^n + (u_h^n)^2 - 1.
$$
Using summation by parts in Lemma \ref{lem20180515_1}, we have
\begin{align} \label{eq20180822_6}
J_2 &=  -\frac{k}{\epsilon}\sum_{n=1}^\ell
(d_t (M^n\theta^n), \Delta \theta^{n-1})_h +
\frac{1}{\epsilon}( M^l \theta^l, \Delta \theta^l)_h \\ 
& \leq \frac{Ck}{\epsilon} \sum_{n=1}^\ell \|d_t(M^n\theta^n)\|_{0,2}
|\theta|_{2,2,h} + \frac{C}{\epsilon}\|M^l\theta^l\|_{0,2}
|\theta^l|_{2,2,h}. \notag
\end{align}
Since $d_t u_h^n = d_t(P_hu^n) - d_t\theta^n$, a direct calculation
shows that  
\begin{align*}
d_t(M^n \theta^n) &= \theta^n d_t M^n + M^{n-1}d_t \theta^n \\
& = M^{n-1}d_t \theta^n + \theta^n(P_h u^{n} + P_h u^{n-1})
  d_t(P_hu^n) \\ 
&~~~ + \theta^n u_h^n d_t(P_hu^n) + \theta^n P_h u^{n-1}d_t(P_hu^n) 
 - \theta^n P_hu^{n-1} d_t \theta^n \\ 
&~~~ + \theta^n(u_h^n+u_h^{n-1})d_t(P_hu^n) - \theta^n(u_h^n +
    u_h^{n-1})d_t\theta^n \\
& = (M^{n-1} -\theta^nP_h u^{n-1} - \theta^n(u_h^n +
      u_h^{n-1}))d_t\theta^n \\
&~~~ + (P_hu^n + 2P_hu^{n-1} + 2u_h^n + u_h^{n-1})\theta^n
d_t(P_hu^n).
\end{align*}
Using the $L^2(H^2)$ error estimate \eqref{eq20180820_1} and the
assumption on the $L^\infty$ bound of $u_h^n$, we get  
\begin{align} \label{eq20180822_7}
& \quad ~\frac{Ck}{\epsilon}\sum_{n=1}^\ell
\|d_t(M^n\theta^n)\|_{0,2}|\theta^n|_{2,2,h} \\
& \leq C\epsilon^{-2\gamma_1 - 1} k\sum_{n=1}^\ell
\|d_t\theta^n\|_{0,2}|\theta^n|_{2,2,h} 
+ C\epsilon^{-\gamma_1 - 1} k\sum_{n=1}^\ell \|\theta^n d_t(P_h
    u)\|_{0,2}|\theta^n|_{2,2,h} \notag \\
& \leq \frac{k}{8}\sum_{n=1}^\ell \|d_t \theta^n\|_{0,2}^2 +
C\epsilon^{-4\gamma_1 - 2} k\sum_{n=1}^\ell |\theta|_{2,2,h}^2 +
C\epsilon^{2\gamma_1}k\sum_{n=1}^\ell \|\theta d_t(P_h u)\|_{0,2}^2
\notag \\
& \leq 
\frac{k}{8}\sum_{n=1}^\ell \|d_t \theta^n\|_{0,2}^2 +
C\epsilon^{-4\gamma_1 - 3}(\tilde{\rho}_2(\epsilon)|\ln h|^2h^2 +
    \tilde{\rho}_3(\epsilon)|\ln h|k^2) \notag \\
&~~~ + 
C \epsilon^{2\gamma_1-\max\{ 2\sigma_1 + \frac{13}{2}, 2\sigma_3+\frac72,
  2\sigma2+4, 2\sigma_4 \}-1}(\tilde{\rho}_2(\epsilon)|\ln h|^2 h^2 +
      \tilde{\rho}_3(\epsilon)|\ln h| k^2), \notag
\end{align}
where by \eqref{eq2.15_add} and 
the $L^\infty(L^2)$ error estimate \eqref{eq20180820_1},
$$ 
\begin{aligned}
&\quad ~k\sum_{n=1}^\ell \|\theta d_t(P_h u)\|_{0,2}^2 \\
 &\leq \supn\|\theta^n\|_{0,2}^2 \frac{1}{k} \|\int_{t_{n-1}}^{t_n}
 (P_hu)_t \,\rd s\|_{0,\infty}^2 \\
& \leq \supn\|\theta^n\|_{0,2}^2 \int_{0}^{T}
\|(P_hu)_t\|_{0,\infty}^2 \,\rd s \\
& \leq C \epsilon^{-\max\{ 
2\sigma_1 + \frac{13}{2}, 2\sigma_3+\frac72, 2\sigma2+4, 2\sigma_4
\}-1}(\tilde{\rho}_2(\epsilon)|\ln h|^2 h^2 +
    \tilde{\rho}_3(\epsilon)|\ln h| k^2).
\end{aligned}
$$ 
And the second term on the right hand side of \eqref{eq20180822_6} can
be bounded by 
\begin{align} \label{eq20180822_8} 
\frac{C}{\epsilon} \|M^l\theta^l\|_{0,2}|\theta^l|_{2,2,h} 
&\leq C^{-4\gamma_1-3}\|\theta^l\|_{0,2}^2 +
\frac{\epsilon}{8}a_h(\theta^l, \theta^l) \\
&\leq C\epsilon^{-4\gamma_1 - 3}(\tilde{\rho}_2(\epsilon)|\ln h|^2 h^2 +
    \tilde{\rho}_3(\epsilon)|\ln h|k^2)
   + \frac{\epsilon}{8}a_h(\theta^l, \theta^l).  \notag 
\end{align} 
Combining \eqref{eq20180822_7} and \eqref{eq20180822_8}, we obtain
the bound for $J_2$:
\begin{align} \label{eq20180822_9}
J_2 &\leq 
\frac{k}{8}\sum_{n=1}^\ell \|d_t \theta^n\|_{0,2}^2 
   + \frac{\epsilon}{8}a_h(\theta^l, \theta^l) +
C\epsilon^{-4\gamma_1 - 3}(\tilde{\rho}_2(\epsilon)|\ln h|^2 h^2 +
    \tilde{\rho}_3(\epsilon)|\ln h|k^2) \\
&~~~ + 
C \epsilon^{2\gamma_1-\max\{ 2\sigma_1 + \frac{13}{2}, 2\sigma_3+\frac72,
  2\sigma_2+4, 2\sigma_4 \}-1}(\tilde{\rho}_2(\epsilon)|\ln h|^2 h^2 +
      \tilde{\rho}_3(\epsilon)|\ln h| k^2). \notag
\end{align}

\paragraph{$\bullet$ Estimate of $J_3$}
Notice that $\theta^n \in S_E^h$ and  
$$ 
\int_{E} \llbracket\nabla \theta^n\rrbracket\,\rd S = 0 \qquad \forall E \in
\mathcal{E}_h.
$$ 
Using summation by parts in Lemma \ref{lem20180515_1}, Lemma 2.2 in
\cite{elliott1989nonconforming} and inverse inequality, we have
\begin{align*}
J_3 &= \frac{k}{\epsilon}\sum_{n=1}^\ell\sum_{E\in \mathcal{E}_h} 
(d_t \{M^n\theta^n\}, \llbracket\nabla \theta^{n-1}\rrbracket)_E -
\frac{1}{\epsilon}\sum_{E\in \mathcal E_h}(\{M^\ell \theta^{\ell}\}, \llbracket\nabla \theta^{\ell}\rrbracket)_E \\
& \leq \frac{Ck}{\epsilon} \sum_{n=1}^\ell \|d_t(M^n\theta^n)\|_{0,2}
|\theta|_{2,2,h} + \frac{C}{\epsilon}\|M^\ell\theta^\ell\|_{0,2}
|\theta^l|_{2,2,h}.
\end{align*} 
Hence, $J_3$ has the same bound as $J_2$.

\paragraph{$\bullet$ Estimate of $J_4$} Since $P_hu$ and $u_h$ are
continuous at vertexes of $\mathcal{T}_h$, thanks to Lemma 2.6 in
\cite{elliott1989nonconforming}, we have 
\begin{align} \label{eq20180823_2}
J_4 & \leq \frac{Ck}{\epsilon} \sum_{n=1}^\ell h|M^n\theta^n|_{2,2,h}
|d_t\theta^n|_{1,2,h}\\
&\leq \frac{Ck}{\epsilon} \sum_{n=1}^\ell |M^n\theta^n|_{2,2,h}
\|d_t\theta^n\|_{0,2} \notag\\
& \leq \frac{Ck}{\epsilon^2} \sum_{n=1}^\ell |M^n\theta^n|_{2,2,h}^2 +
\frac{k}{8}\sum_{n=1}^\ell \|d_t \theta^n\|_{0,2}^2 \notag.
\end{align} 
Using the piecewise $L^2(H^2)$ estimate given in Theorem
\ref{thm20171007_1}, we have 
\begin{align} \label{eq20180823_1}
& \quad ~\frac{Ck}{\epsilon^2} \sum_{n=1}^\ell |M^n\theta^n|_{2,2,h}^2
\\
& \leq \frac{Ck}{\epsilon^2} \sum_{n=1}^\ell  \left( 
\|M^n\|_{0,\infty}^2 |\theta^n|_{2,2,h}^2 +
|M^n|_{1,4,h}^2|\theta^n|_{1,4,h}^2 +
|M^n|_{2,2,h}^2 \|\theta^n\|_{0,\infty}^2 
\right) \notag \\
& \leq \frac{C}{\epsilon^2}\supn
\|M^n\|_{2,2,h}^2 k\sum_{n=1}^\ell \|\theta^n\|_{2,2,h}^2 \notag \\
& \leq C(\epsilon^{-4\gamma_2-2} + \epsilon^{-\max\{2\sigma_1+5,
    2\sigma_3+2\}-2})(\tilde{\rho}_2(\epsilon)|\ln h|^2 h^2 +
      \tilde{\rho}_3(\epsilon)|\ln h|k^2),\notag
\end{align}
where by \eqref{eq2.13_add} and the fact that $\|u_h^n\|_{2,2,h} \leq
C\epsilon^{-\gamma_2}$ (c.f. \cite[Theorem 3.14]{li2017error}) 
$$ 
\begin{aligned}
\|M^n\|_{2,2,h} &\leq C (\|(P_hu^n)^2\|_{2,2,h} + \|u_h^n
P_hu^n\|_{2,2,h} + \|(u_h^n)^2\|_{2,2,h}) \\
& \leq C(\|P_hu^n\|_{2,2,h} + \|P_hu^n\|_{1,4,h}^2 +
    \|u_h\|_{0,\infty}\|u_h^n\|_{2,2,h} + \|u_h^n\|_{1,4,h}^2\\
&~~~ + \|u_h^n\|_{2,2,h} + \|u_h^n\|_{0,\infty}\|P_hu^n\|_{2,2,h} +
\|u_h^n\|_{1,4,h}\|P_hu^n\|_{1,4,h}) \\
& \leq C(\epsilon^{-2\gamma_2} + \epsilon^{-\max\{2\sigma_1+5,
    2\sigma_3 + 2\}}).
\end{aligned}
$$ 

\underline{Piecewise $L^\infty(H^2)$ error estimate}: Taking
\eqref{eq20180820_2}, \eqref{eq20180820_3}, \eqref{eq20180822_4},
\eqref{eq20180822_9} and \eqref{eq20180823_2} into
\eqref{eq20180214_2}, we obtain  
\begin{align}
&\quad ~ \frac{k}{8} \sum_{n=1}^\ell \|d_t\theta^n\|_{L^2}^2 +
\frac{\epsilon}{8} a_h(\theta^\ell, \theta^\ell) + 
\frac{\epsilon k^2}{2} \sum_{n=1}^\ell a_h(d_t\theta^n,d_t\theta^n) \\
& \leq  
C (\epsilon^4\rho_3(\epsilon) + \epsilon^{-6}\rho_4(\epsilon)) h^2
+ C\rho_5(\epsilon)|\ln h|^2 h^2 + C\rho_3(\epsilon)k^2
\notag \\
&~~~ + C(\epsilon^{-2\sigma_1 - 1}\rho_3(\epsilon) +
\epsilon^{-4}\rho_0(\epsilon)\rho_4(\epsilon) +
\epsilon^{-2\sigma_1-5}\rho_5(\epsilon) +
\epsilon^{-1}\tilde{\rho}_2(\epsilon) ) |\ln h|^2 h^2 \notag \\
&~~~ + C\epsilon^{-1}\tilde{\rho}_3(\epsilon)|\ln h|k^2 +
C\epsilon^{-4\gamma_1 - 3}(\tilde{\rho}_2(\epsilon)|\ln h|^2 h^2 +
        \tilde{\rho}_3(\epsilon)|\ln h|k^2) \notag \\
&~~~ + C \epsilon^{-\max\{ 2\sigma_1 + \frac{13}{2},
2\sigma_3+\frac72, 2\sigma_2+4,
2\sigma_4\}-1}(\tilde{\rho}_2(\epsilon)|\ln h|^2 h^2 +
    \tilde{\rho}_3(\epsilon)|\ln h| k^2) \notag \\
&~~~ + C(\epsilon^{-4\gamma_2-2} + \epsilon^{-\max\{2\sigma_1+5,
2\sigma_3+2\}-2})(\tilde{\rho}_2(\epsilon)|\ln h|^2 h^2 +
\tilde{\rho}_3(\epsilon)|\ln h|k^2). \notag
\end{align}
Then the theorem can be proved by simplifying the coefficients
according to the definitions of $\rho_i(\epsilon)$ and
$\tilde{\rho}_i(\epsilon)$.
\end{proof}

\begin{remark} \label{rmk20180823_1}
If the summation by part for time and integration by part for space
techniques are not employed simultaneously, one can only obtain a 
coarse estimate 
\begin{align*}
&\quad ~\|\theta^\ell\|_{2,2,h}^2 + k\sum_{n=1}^{\ell}\|d_t\theta^n\|_{L^2}^2
+\epsilon k^2 \sum_{n=1}^{\ell}a_h(d_t\theta^n, d_t \theta^n) \\
&\le Ck^{-\frac12}(\epsilon^{-\gamma_4}|\ln h|^2h^2+\epsilon^{-\gamma_5}|\ln h|k),
\end{align*}
where $\gamma_4, \gamma_5$ denote some positive constants. 
\end{remark}

Finally, using \eqref{eq20180819_5}, Theorem \ref{thm20180214_4} and
the Sobolev embedding theorem, we can prove the desired
$L^\infty(L^\infty)$ error estimate.

\begin{theorem}\label{thm20180611_add2}
Assume $u$ is the solution of
\eqref{eq20170504_1}--\eqref{eq20170504_5}, $u_h^n$ is the numerical
solution of scheme \eqref{eq20170504_11}--\eqref{eq20170504_12}.
Under the mesh constraints in Theorem 3.15 in \cite{li2017error} and \eqref{mesh_cond}, we
have the $L^\infty(L^\infty)$ error estimate 
\begin{align} \label{eq20180823_3}
\|u(t_n) - u_h^n\|_{L^{\infty}}\le C|\ln
h|^{\frac12}((\tilde{\rho}_4(\epsilon))^{\frac12}|\ln h|^{\frac12}h + (\tilde{\rho}_5(\epsilon))^{\frac12}k)
\quad \forall 1 \leq n \leq \ell.
\end{align}
\end{theorem}

\begin{remark}
The mesh constraints in Theorem 3.15 in \cite{li2017error} and \eqref{mesh_cond} can be achieved by 
$h = C\epsilon^{p_1}$ and $k = C\epsilon^{p_2}$ for certain positive $p_1, p_2$. Hence, the $|\ln h| k^2$ decreases 
asymptoticly as $k^2$ when $\epsilon$ goes to zero. 
\end{remark}

\section{Convergence of the Numerical Interface}\label{sec5}
In this section, we prove that the numerical interface defined as the
zero level set of the Morley element interpolation of the solution
$U^n$ converges to the moving interface of the Hele-Shaw problem under
the assumption that the Hele-Shaw problem has a unique global (in
time) classical solution. We first cite the following convergence
result established in \cite{alikakos1994convergence}.

\begin{theorem}\label{thm4.1}
Let $\Omega$ be a given smooth domain and $\Gamma_{00}$ be a smooth
closed hypersurface in $\Omega$. Suppose that the Hele-Shaw problem
starting from $\Gamma_{00}$ has a unique smooth solution
$\bigl(w,\Gamma:=\bigcup_{0\leq t\leq T}(\Gamma_t\times\{t\}) \bigr)$
in the time interval $[0,T]$ such that $\Gamma_t\subseteq\Omega$\
\,for all $t\in[0,T]$.  Then there exists a family of smooth functions
$\{u_{0}^{\epsilon}\}_{0<\epsilon\leq 1}$ which are uniformly bounded
in $\epsilon\in(0,1]$ and $(x,t)\in \overline{\Omega}_T$, such that if
    $u^{\epsilon}$ solves the Cahn-Hilliard problem
    \eqref{eq20170504_1}--\eqref{eq20170504_3}, then
\begin{itemize}
\item[\rm (i)] $\displaystyle{\lim_{\epsilon\rightarrow 0}} 
u^{\eps}(x,t)= \begin{cases}
1 &\qquad \mbox{if}\, (x,t)\in \mathcal{O}\\
-1 &\qquad \mbox{if}\, (x,t)\in \mathcal{I}
\end{cases}
\,\mbox{ uniformly on compact subsets}$,
where $\mathcal{I}$ and $\mathcal{O}$ stand for the ``inside" and ``outside" of $\Gamma$;
\item[\rm (ii)] $\displaystyle{\lim_{\epsilon\rightarrow 0}}
\bigl( \epsilon^{-1} f(u^{\epsilon})-\epsilon\Delta u^{\epsilon} \bigr)(x,t)=-w(x,t)$ uniformly on 
$\overline{\Omega}_T$.
\end{itemize}
\end{theorem}

We are now ready to state the first main theorem of this section.

\begin{theorem}\label{thm4.2}
Let $\{\Gamma_t\}_{t\geq0}$ denote the zero level set of the Hele-Shaw problem 
and  $U_{\epsilon,h,k}(x,t)$ denotes the piecewise linear 
interpolation in time of the numerical solution $u_h^n$, namely,
\begin{align}
U_{\epsilon,h,k}(x,t):=\frac{t-t_{n-1}}{k}u_h^{n}(x)+\frac{t_{n}-t}{k}u_h^{n-1}(x),	\label{eq4.1}
\end{align}
for $t_{n-1}\leq t\leq t_{n}$ and $1\leq n\leq M$.
Then, under the mesh and starting value constraints of Theorem \ref{thm20180214_4} 
and $k=O(h^q)$ with $0<q<1$, we have
\begin{itemize}
\item[\rm (i)] $U_{\epsilon,h,k}(x,t) \stackrel{\eps\searrow 0}{\longrightarrow} 1$ uniformly 
on compact subset of $\mathcal{O}$,
\item[\rm (ii)] $U_{\epsilon,h,k}(x,t) \stackrel{\eps\searrow 0}{\longrightarrow} -1$ uniformly on 
compact subset of $\mathcal{I}$.
\end{itemize}
\end{theorem}

\smallskip
\begin{proof}
For any compact set $A\subset\mathcal{O}$ and for any $(x,t)\in A$, we have
\begin{align} \label{eq4.4}
|U_{\epsilon,h,k}-1|&\leq |U_{\epsilon,h,k}-u^{\epsilon}(x,t)|+|u^{\epsilon}(x,t)-1| \\
&\leq |U_{\epsilon,h,k}-u^{\epsilon}(x,t)|_{L^{\infty}(\Omega_T)}+|u^{\epsilon}(x,t)-1|.\nonumber
\end{align}
Theorem \ref{thm20180611_add2} infers that
\begin{equation}\label{eq4.5}
|U_{\epsilon,h,k}-u^{\epsilon}(x,t)|_{L^{\infty}(\Omega_T)}\leq C(\tilde{\rho}_6(\epsilon))^{\frac12}h^q|\ln h|.
\end{equation}
where $\tilde{\rho}_6(\epsilon)=\max\{\tilde{\rho}_4(\epsilon),\tilde{\rho}_5(\epsilon)\}.$

The first term on the right-hand side of \eqref{eq4.4} tends to $0$ when $\epsilon\searrow 0$
(note that $h,k\searrow 0$, too). The second term converges uniformly to $0$ on the compact set $A$, 
which is ensured by (i) of Theorem \ref{thm4.1}.  Hence, the assertion (i) holds.

To show (ii), we only need to replace $\mathcal{O}$ by $\mathcal{I}$ and $1$ by $-1$ in the above proof.
\end{proof}

The second main theorem addresses the convergence 
of numerical interfaces.

\begin{theorem}\label{thm4.3}
Let $\Gamma_t^{\epsilon,h,k}:=\{x\in\Omega;\, U_{\epsilon,h,k}(x,t)=0\}$
be the zero level set of\ \,$U_{\epsilon,h,k}(x,t)$, then under the 
assumptions of Theorem \ref{thm4.2}, we have
\[
\sup_{x\in\Gamma_t^{\epsilon,h,k}} \mbox{\rm dist}(x,\Gamma_t)
\stackrel{\epsilon\searrow 0}{\longrightarrow} 0 \quad\mbox{uniformly on $[0,T]$}.
\]
\end{theorem}

\begin{proof}
For any $\eta\in(0,1)$,  define the tabular neighborhood $\mathcal{N}_{\eta}$ of width 
$2\eta$ of $\Gamma_t$ 
\begin{equation}\label{eq4.8}
\mathcal{N}_{\eta}:=\{(x,t)\in\Omega_T;\, \mbox{\rm dist}(x,\Gamma_t)<\eta\}.
\end{equation}
Let $A$ and $B$ denote the complements of the neighborhood $\mathcal{N}_{\eta}$ in $\mathcal{O}$ 
and $\mathcal{I}$, respectively,
\begin{equation*}
A=\mathcal{O}\setminus\mathcal{N}_{\eta} \qquad\mbox{and}\qquad
B=\mathcal{I}\setminus\mathcal{N}_{\eta}.
\end{equation*}
Note that $A$ is a compact subset outside $\Gamma_t$ and $B$ is a compact subset inside $\Gamma_t$. By Theorem \ref{thm4.2}, there exists ${\epsilon_1}>0$, which only depends on $\eta$, such that 
for any $\epsilon\in (0,{\epsilon_1})$
\begin{align}
&|U_{\epsilon,h,k}(x,t)-1|\leq\eta\quad\forall(x,t)\in A,\label{eq4.9}\\
&|U_{\epsilon,h,k}(x,t)+1|\leq\eta\quad\forall(x,t)\in B.\label{eq4.10}
\end{align}
Now for any $t\in[0,T]$ and $x\in \Gamma_t^{\epsilon,h,k}$, from $U_{\epsilon,h,k}(x,t)=0$ we have
\begin{align}
&|U_{\epsilon,h,k}(x,t)-1|=1\qquad\forall(x,t)\in A,\label{eq4.11}\\
&|U_{\epsilon,h,k}(x,t)+1|=1\qquad\forall(x,t)\in B.\label{eq4.12}
\end{align}
\eqref{eq4.9} and \eqref{eq4.11} imply that $(x,t)$ is not in $A$, and \eqref{eq4.10} 
and \eqref{eq4.12} imply that $(x,t)$ is not in $B$, then $(x,t)$ must lie in the tubular 
neighborhood $\mathcal{N}_{\eta}$. Therefore, for any $\epsilon\in(0,\epsilon_1)$,
\begin{equation}\label{eq4.13}
\sup_{x\in\Gamma_t^{\epsilon,h,k}} \mbox{\rm dist}(x,\Gamma_t) \leq\eta \qquad\mbox{uniformly on $[0,T]$}.
\end{equation}
The proof is complete.
\end{proof}

\section{Numerical experiments}\label{sec6}

In this section, we present two two-dimensional numerical tests 
to gauge the performance of the proposed fully discrete Morley finite
element method for Cahn-Hilliard equation. The square domain $\Omega =
[-1,1]^2$ is used in both tests.

\paragraph{Test 1} Consider the Cahn-Hilliard problem with an ellipse
initial interface determined by $\Gamma_0: \frac{x^2}{0.36} +
\frac{y^2}{0.04} = 0$.  The initial condition is chosen to have the
form $u_0(x,y) = \tanh(\frac{d_0(x, y)}{\sqrt{2\epsilon}})$, where
$d_0(x, y)$ denotes the signed distance from $(x,y)$ to the initial
ellipse interface $\Gamma_0$ and $\tanh(t) = (e^t - e^{-t})/(e^t +
e^{-t})$. 

Figure \ref{fig:ellipse} displays four snapshots at four fixed time
points of the numerical interface with four different $\epsilon$'s.
Here time step size $k = 1\times 10^{-4}$ and space size $h = 0.01$
are used.  They clearly indicate that at each time point the numerical
interface converges to the sharp interface $\Gamma_t$ of the Hele-Shaw
flow as $\epsilon$ tends to zero. Note that this initial condition
may not satisfy the General Assumption (GA) due to the singularity of
the signed distance function. We will adopt a smooth initial condition in the later test.

\begin{figure}[!htbp]
\centering 
\includegraphics[width=0.45\textwidth]{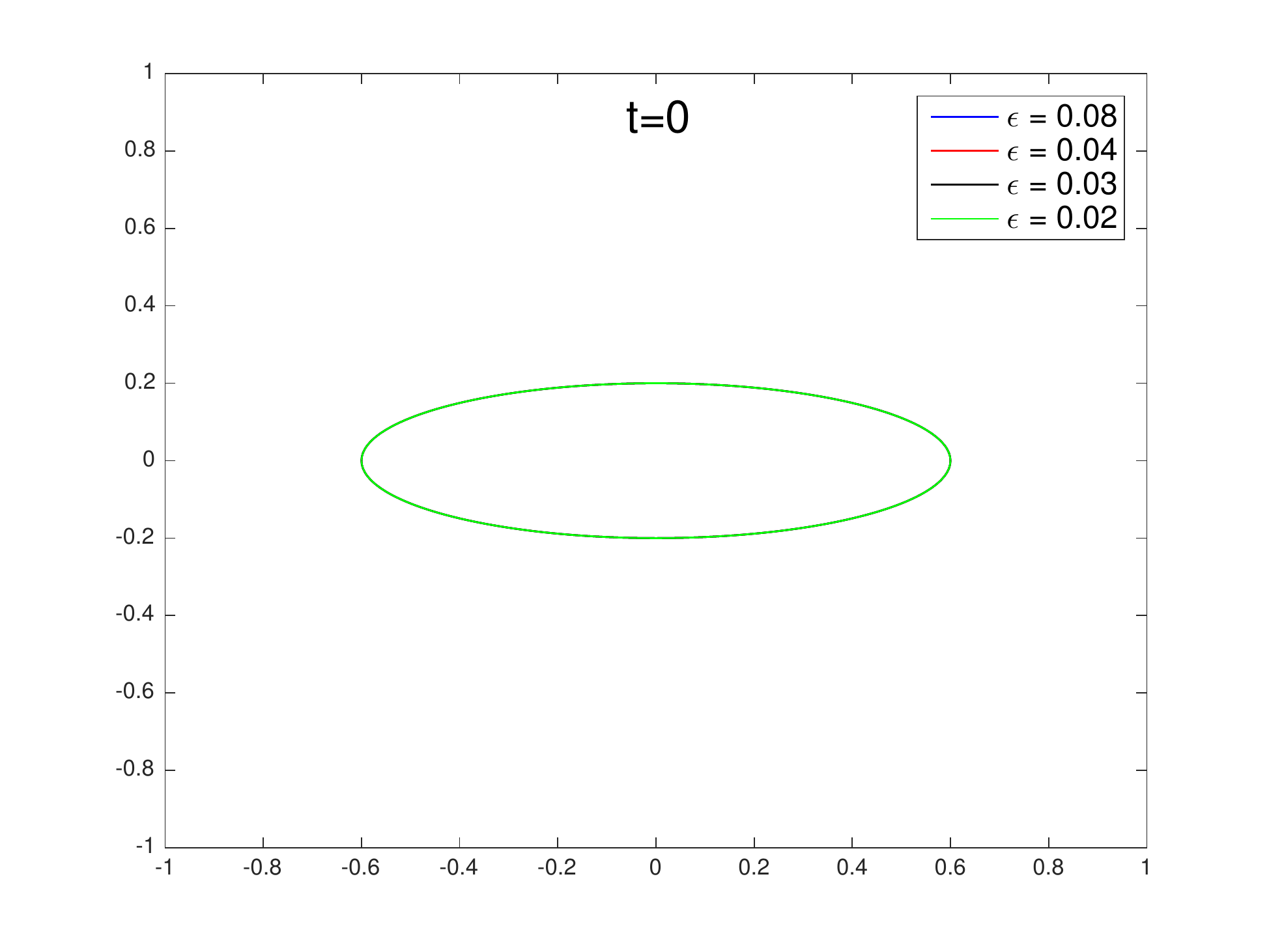} 
\includegraphics[width=0.45\textwidth]{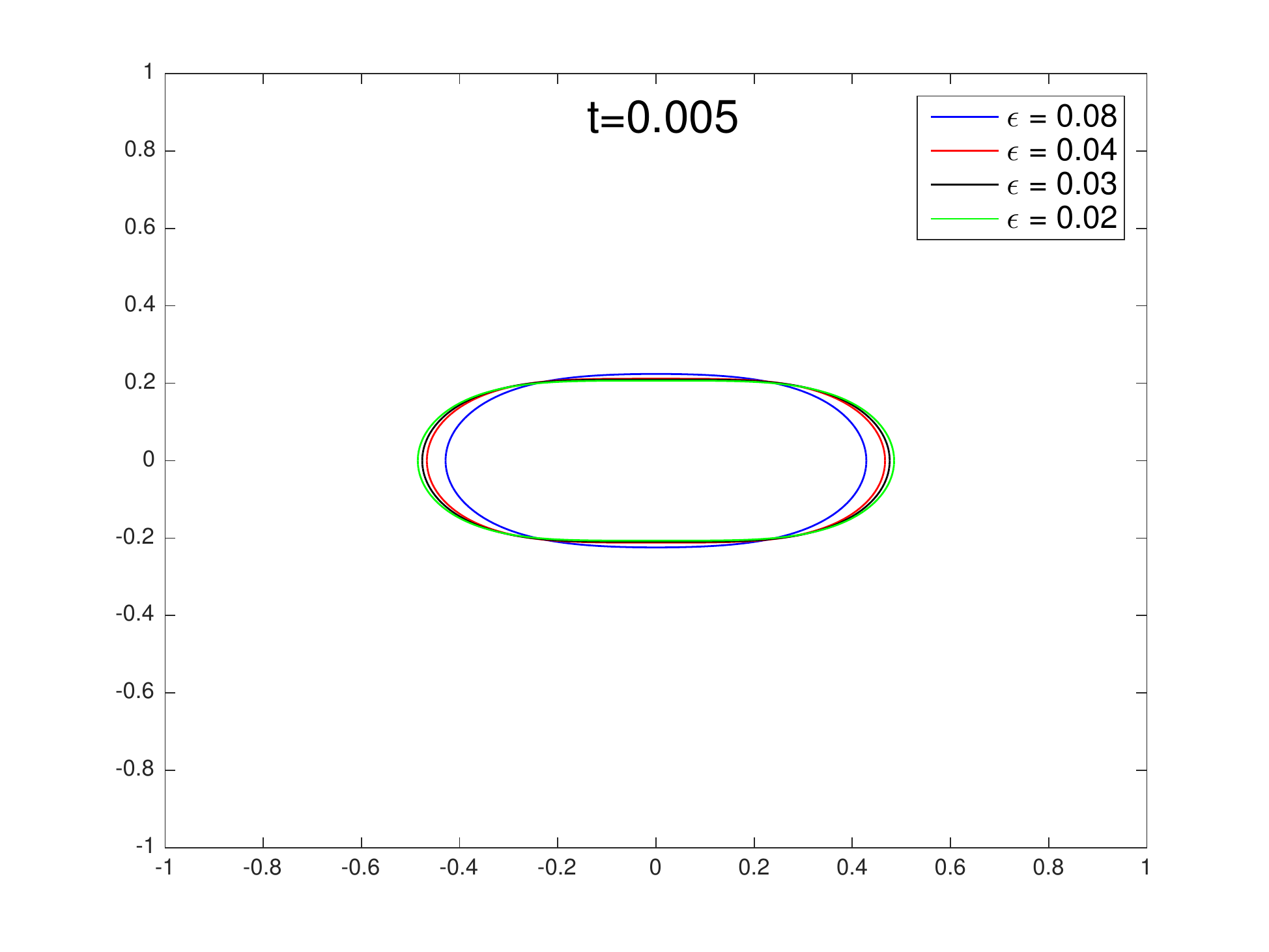} 
\\
\includegraphics[width=0.45\textwidth]{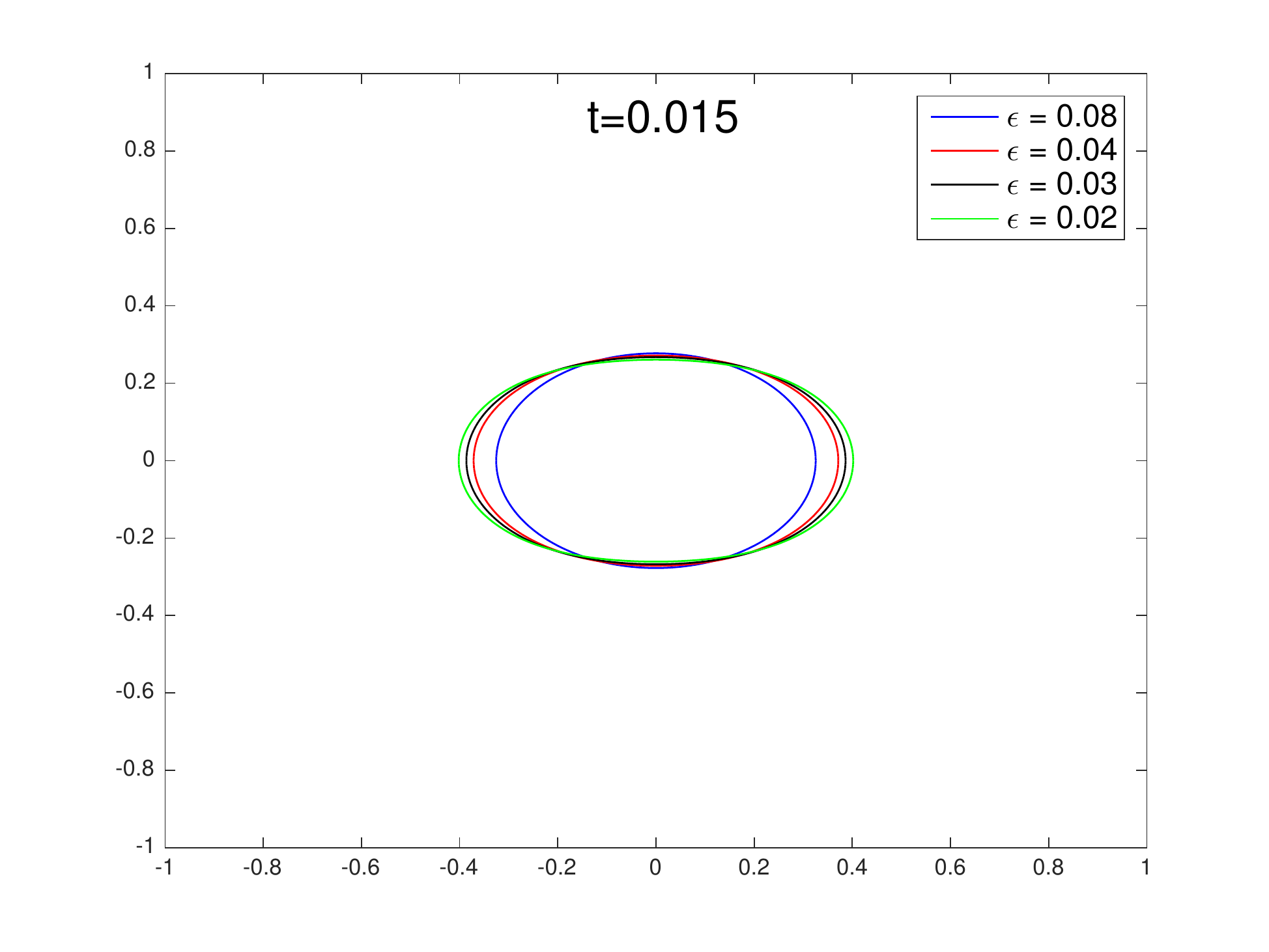} 
\includegraphics[width=0.45\textwidth]{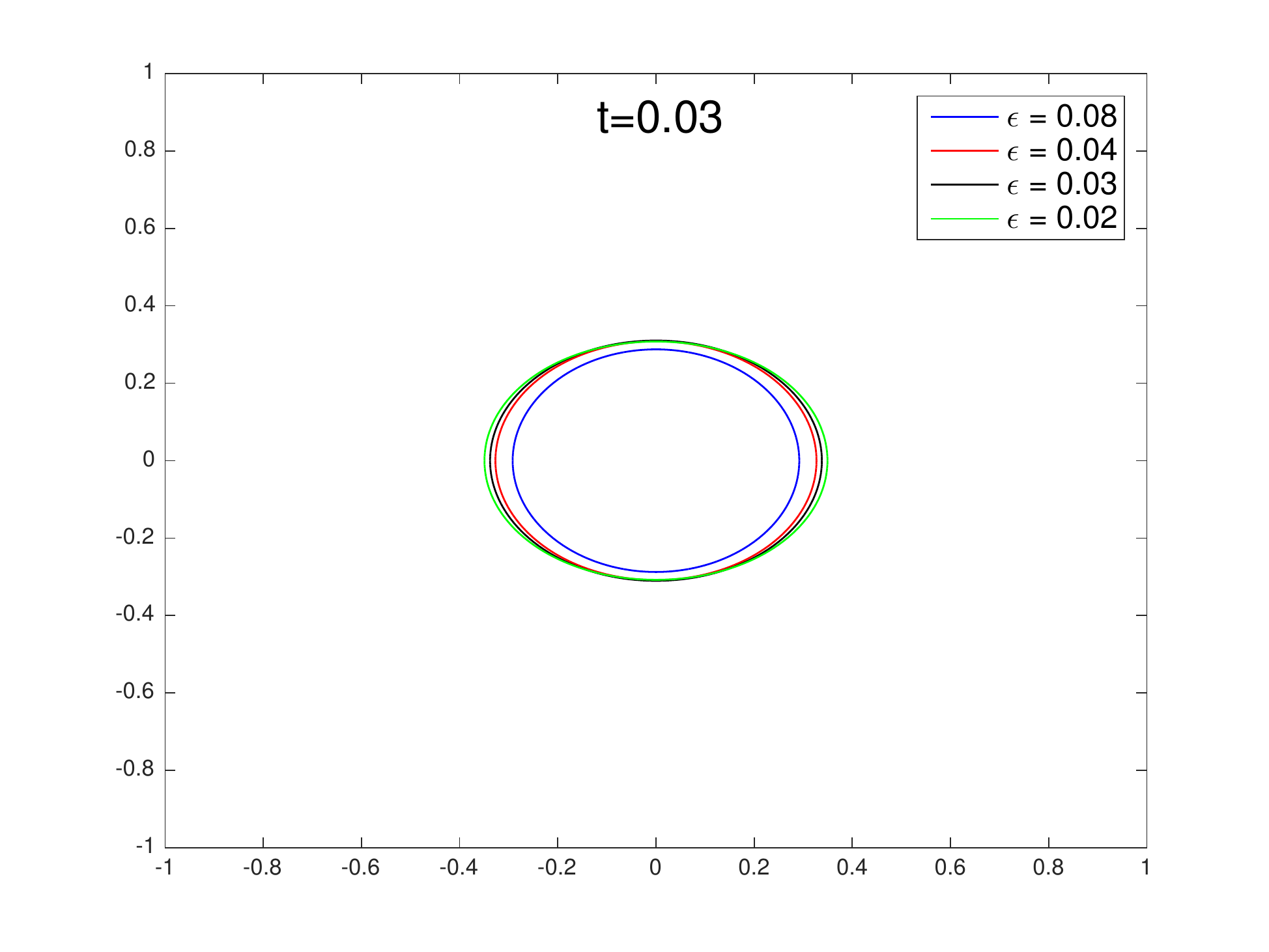} 
\caption{Test 1: Snapshots of the zero-level sets of $u^{\epsilon, k}$
  at $t=0,0.005,0.015,0.03$ and $\epsilon = 0.08,0.04,0.03,0.02$.}
\label{fig:ellipse}
\end{figure}

\paragraph{Test 2}
Consider the following initial condition, which is also adopted in
\cite{feng2008posteriori},
$$
u_0(x,y) = \tanh\Big( ((x-0.3)^2 + y^2 - 0.25^2)/\epsilon \Big) 
\tanh\Big( ((x+0.3)^2 + y^2 - 0.3^2)/\epsilon \Big).
$$

Table
\ref{tab:error1} and \ref{tab:error2} show the errors of spatial
$L^2$, $H^1$ and $H^2$ semi-norms and the rates of convergence at $T =
0.0002$ and $T = 0.001$. $\epsilon = 0.08$ is used to generate the
table. $k = 1\times 10^{-5}$ is chosen so that the error in time is
relatively small to the error in space.  The $L^\infty(H^2)$ norm
error is in agreement with the convergence theorem, but
$L^\infty(L^2)$ and $L^\infty(H^1)$ norm errors are one order higher
than our theoretical results. We note that in
\cite{elliott1989nonconforming}, the second order convergence
for both $L^\infty(L^2)$ and $L^\infty(H^1)$ norms are proved, whereas
only $\frac{1}{\epsilon}$-exponential dependence can be derived.  

\begin{table}[!htbp]
\centering 
\footnotesize
\begin{tabular}{|l||c|c||c|c||c|c|}
\hline  
& $L^{\infty}(L^2)$ error & order &   
$L^{\infty}(H^1)$ error & order &   
$L^{\infty}(H^2)$ error & order \\ \hline    
$h=0.2\sqrt{2}$ & 0.079659 & --- &   
1.761563 & --- &   
34.097686 & --- \\ \hline    
$h=0.1\sqrt{2}$ & 0.023142 & 1.7833 &   
0.642870 & 1.4543 &   
21.604986 & 0.6583\\ \hline    
$h=0.05\sqrt{2}$ & 0.007598 & 1.6067 &   
0.183600 & 1.8080&   
11.783724 & 0.8746\\ \hline    
$h=0.025\sqrt{2}$ & 0.002151  & 1.8201 &   
0.048042 & 1.9342&   
6.045416 &  0.9629\\ \hline    
$h=0.0125\sqrt{2}$ & 0.000557 & 1.9501 &   
0.012167 & 1.9813 &   
3.042138  & 0.9908\\ \hline    
\end{tabular}
\caption{Spatial errors and convergence rates of Test 2: $\epsilon =
  0.08$, $k = 1\times 10^{-5}$, $T = 0.0002$.} \label{tab:error1}
\end{table}

\begin{table}[!htbp]
\centering 
\footnotesize
\begin{tabular}{|l||c|c||c|c||c|c|}
\hline  
& $L^{\infty}(L^2)$ error & order &   
$L^{\infty}(H^1)$ error & order &   
$L^{\infty}(H^2)$ error & order \\ \hline    
$h=0.2\sqrt{2}$ & 0.137170 & --- &   
2.469582 & --- &   
43.008910 & --- \\ \hline    
$h=0.1\sqrt{2}$ & 0.032310 & 2.0859 &   
0.710340 & 1.7977 &   
23.320078 & 0.8831\\ \hline    
$h=0.05\sqrt{2}$ & 0.008830 & 1.8715 &   
0.183932 & 1.9493 &   
11.774451 & 0.9859\\ \hline    
$h=0.025\sqrt{2}$ & 0.002349 & 1.9103 &   
0.046810 & 1.9743 &    
5.927408  &  0.9902\\ \hline    
$h=0.0125\sqrt{2}$ & 0.000597 & 1.9746 &   
0.011764 & 1.9924 &   
2.970322 & 0.9968\\ \hline    
\end{tabular}
\caption{Spatial errors and convergence rates of Test 2: $\epsilon =
  0.08$, $k = 1\times 10^{-5}$, $T = 0.001$.} \label{tab:error2}
\end{table}

Figure \ref{fig:2circle} displays six snapshots at six fixed time
points of the numerical interface with four different $\epsilon$.
Again, they clearly indicate that at each time point the numerical
interface converges to the sharp interface $\Gamma_t$ of the Hele-€揝haw
flow as $\epsilon$ tends to zero. 

\begin{figure}[!htbp]
\centering 
\includegraphics[width=0.45\textwidth]{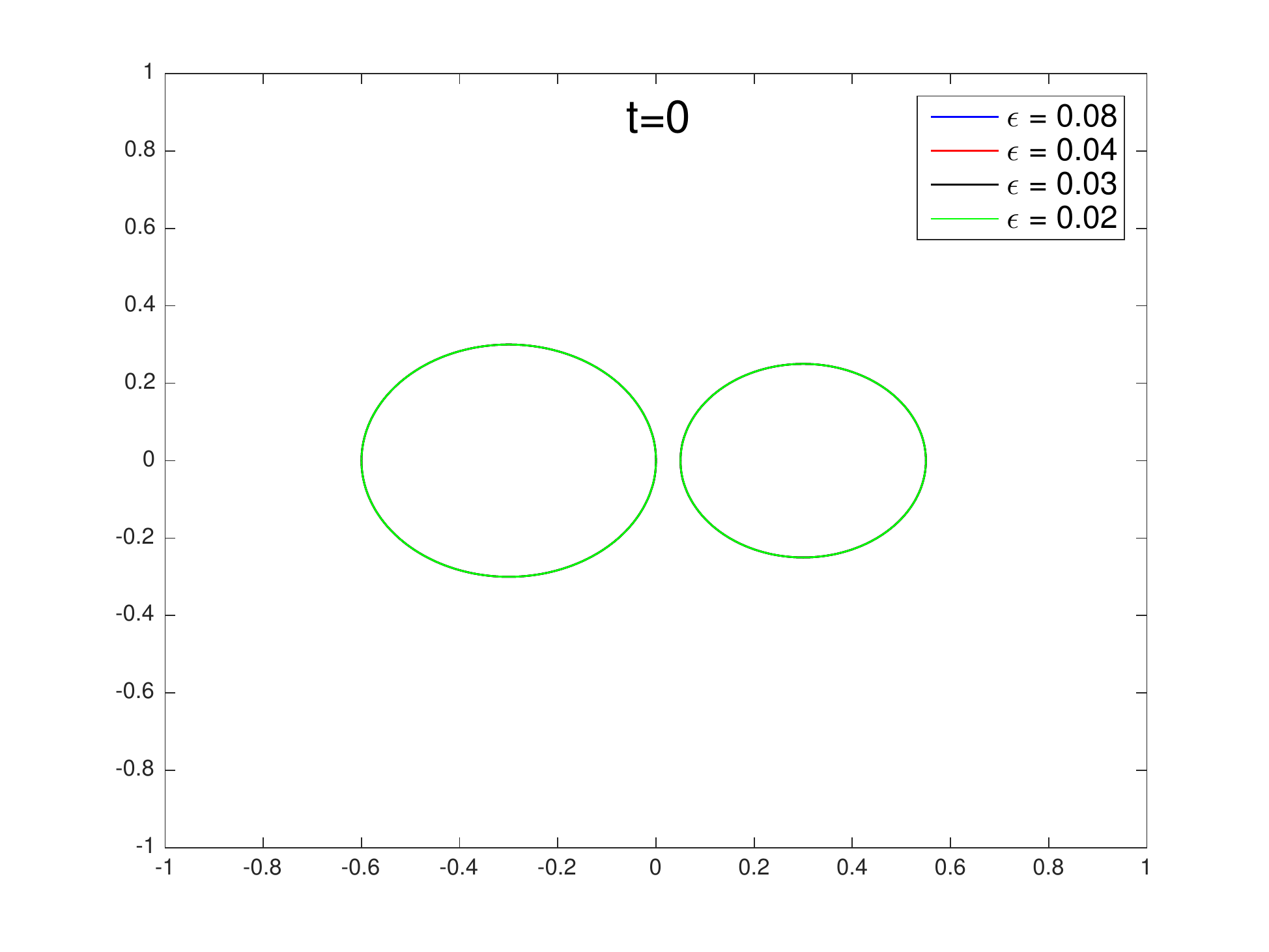} 
\includegraphics[width=0.45\textwidth]{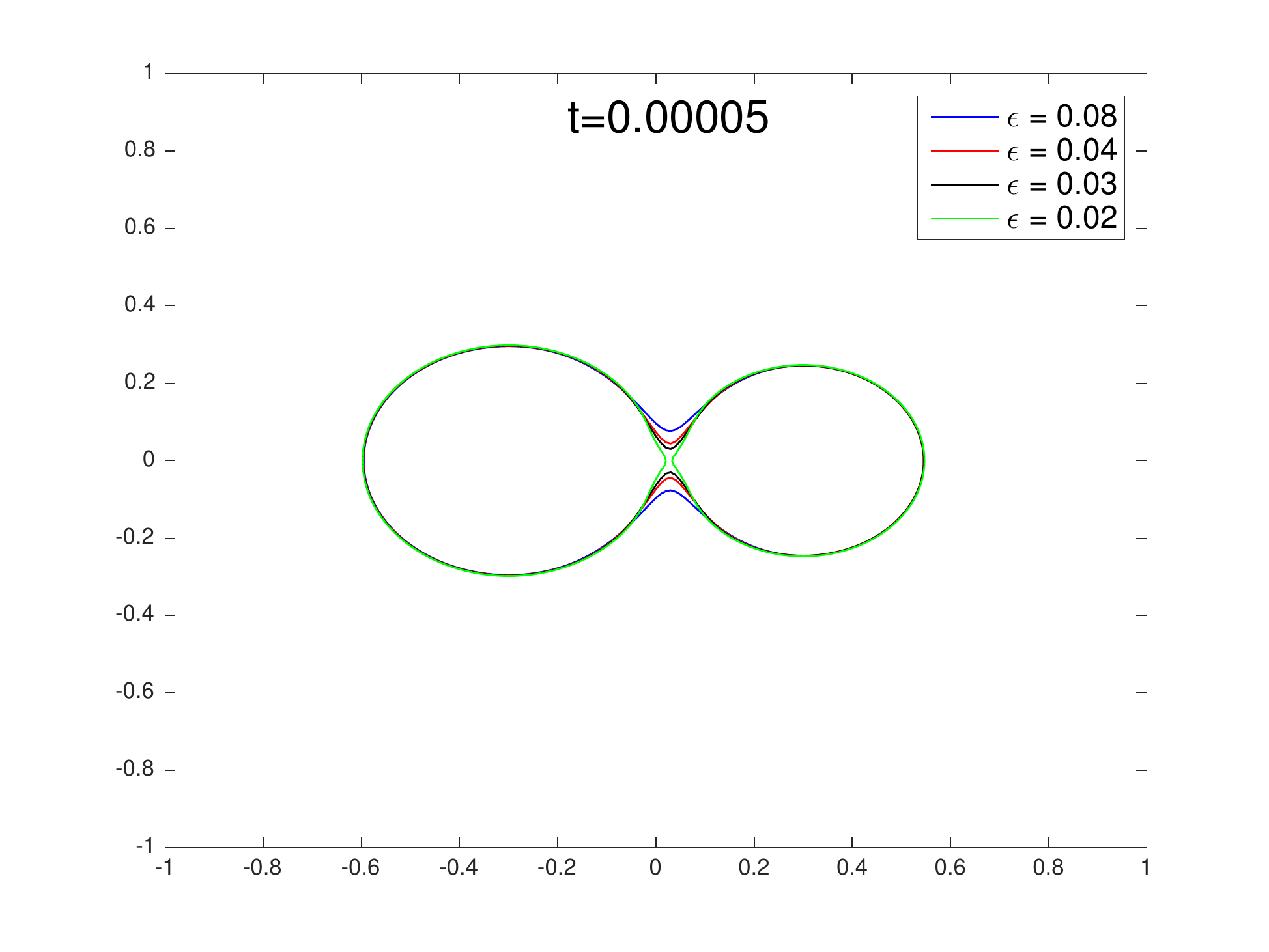} 
\\
\includegraphics[width=0.45\textwidth]{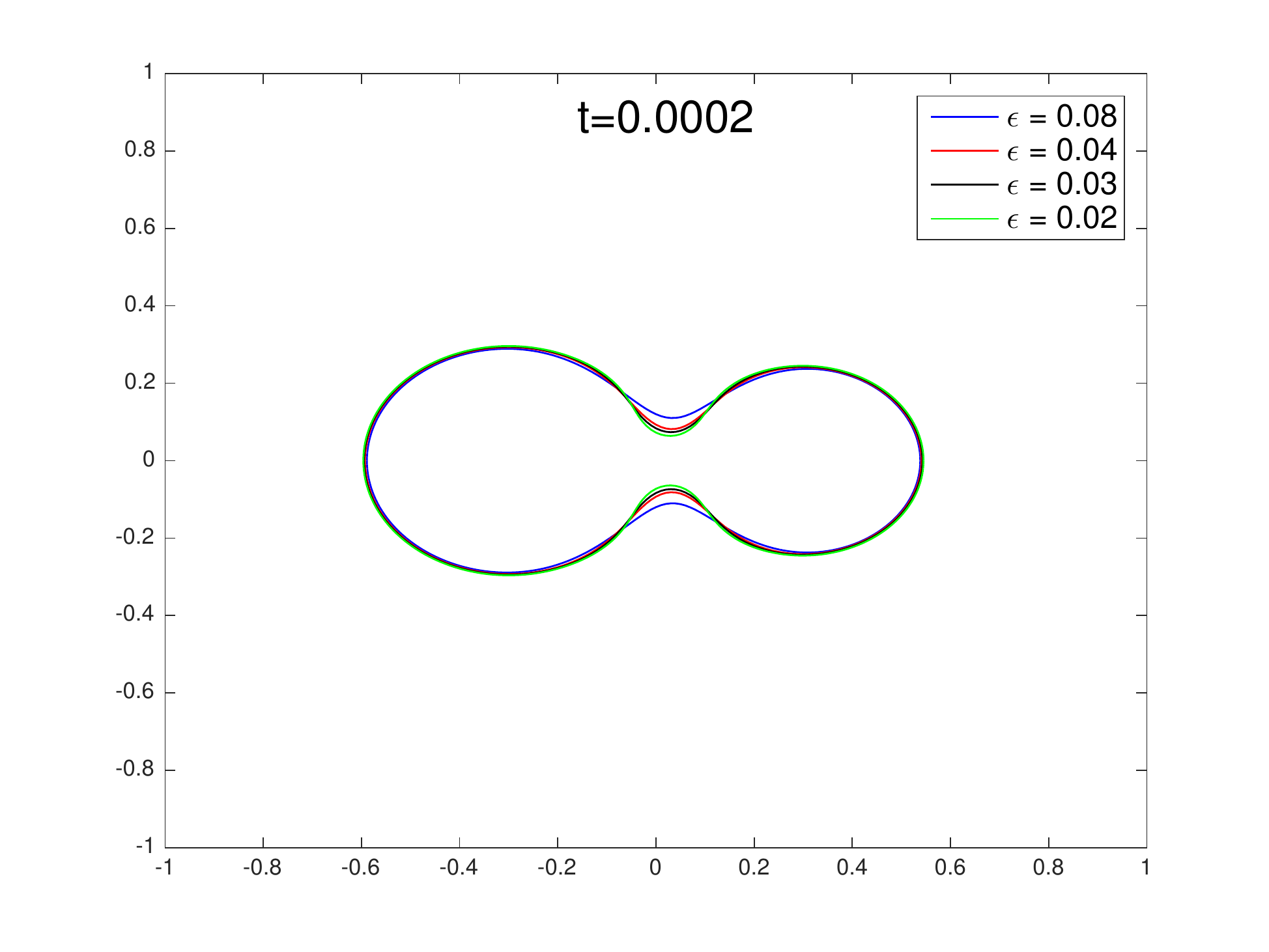} 
\includegraphics[width=0.45\textwidth]{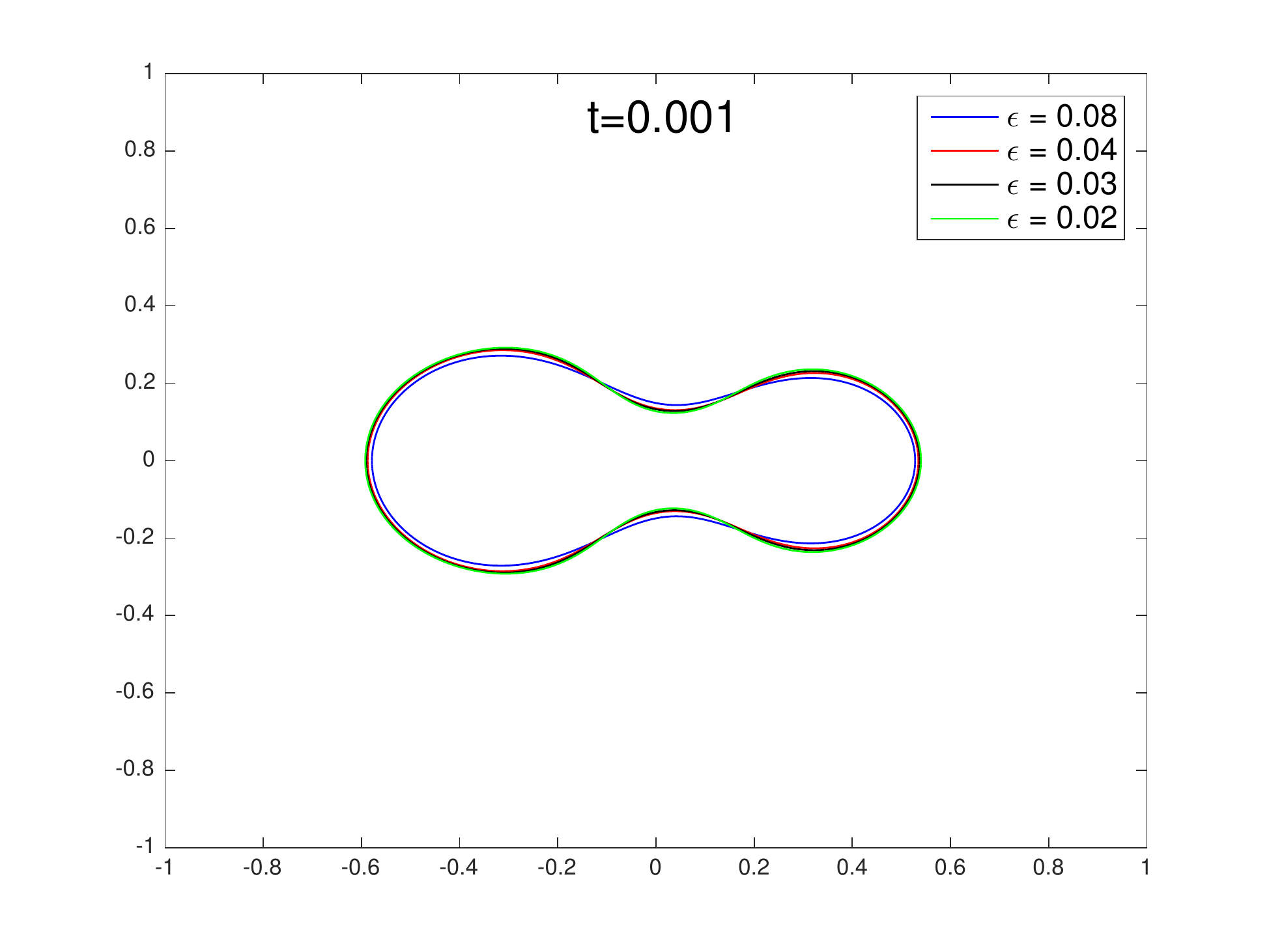} 
\\
\includegraphics[width=0.45\textwidth]{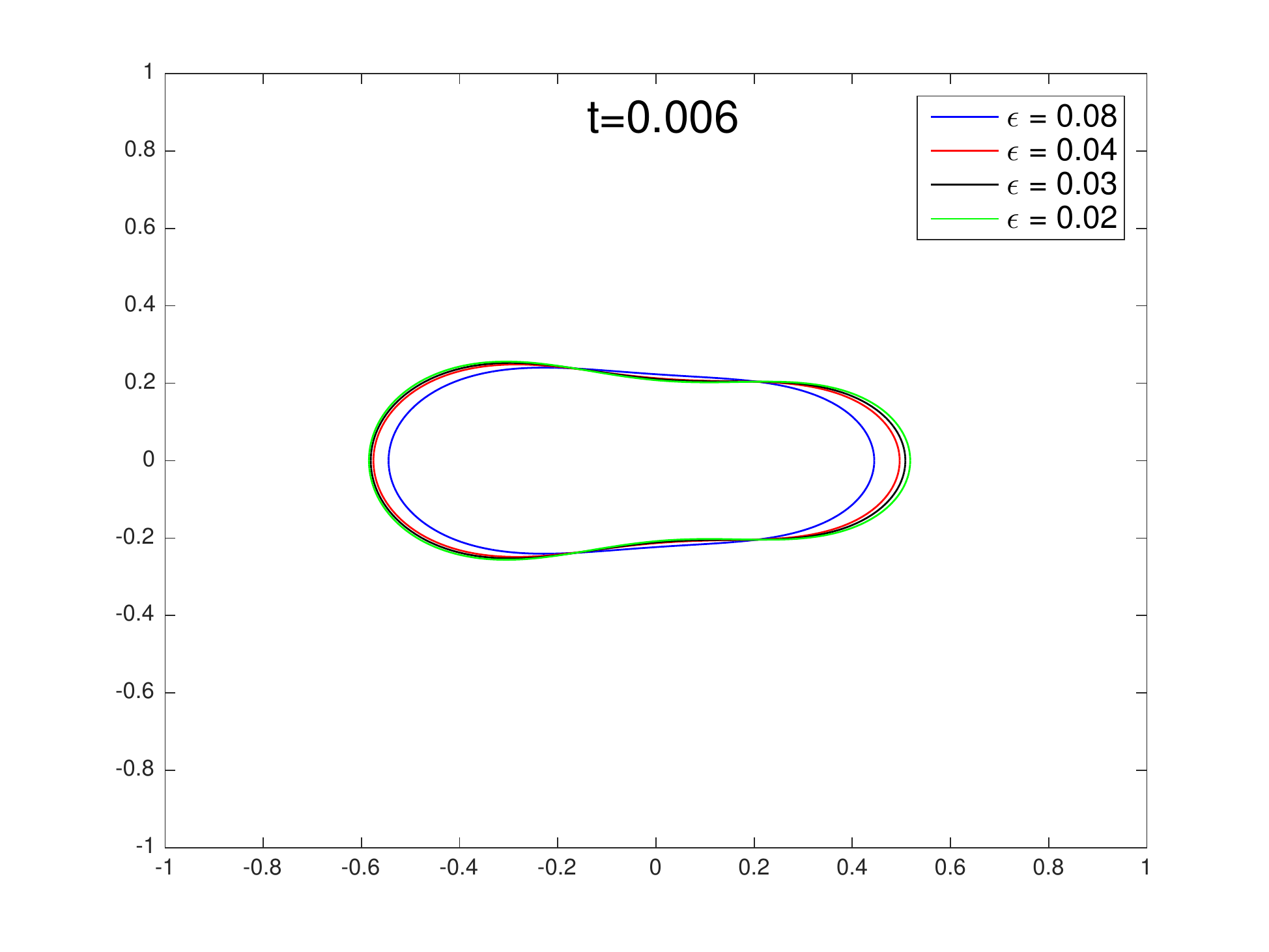} 
\includegraphics[width=0.45\textwidth]{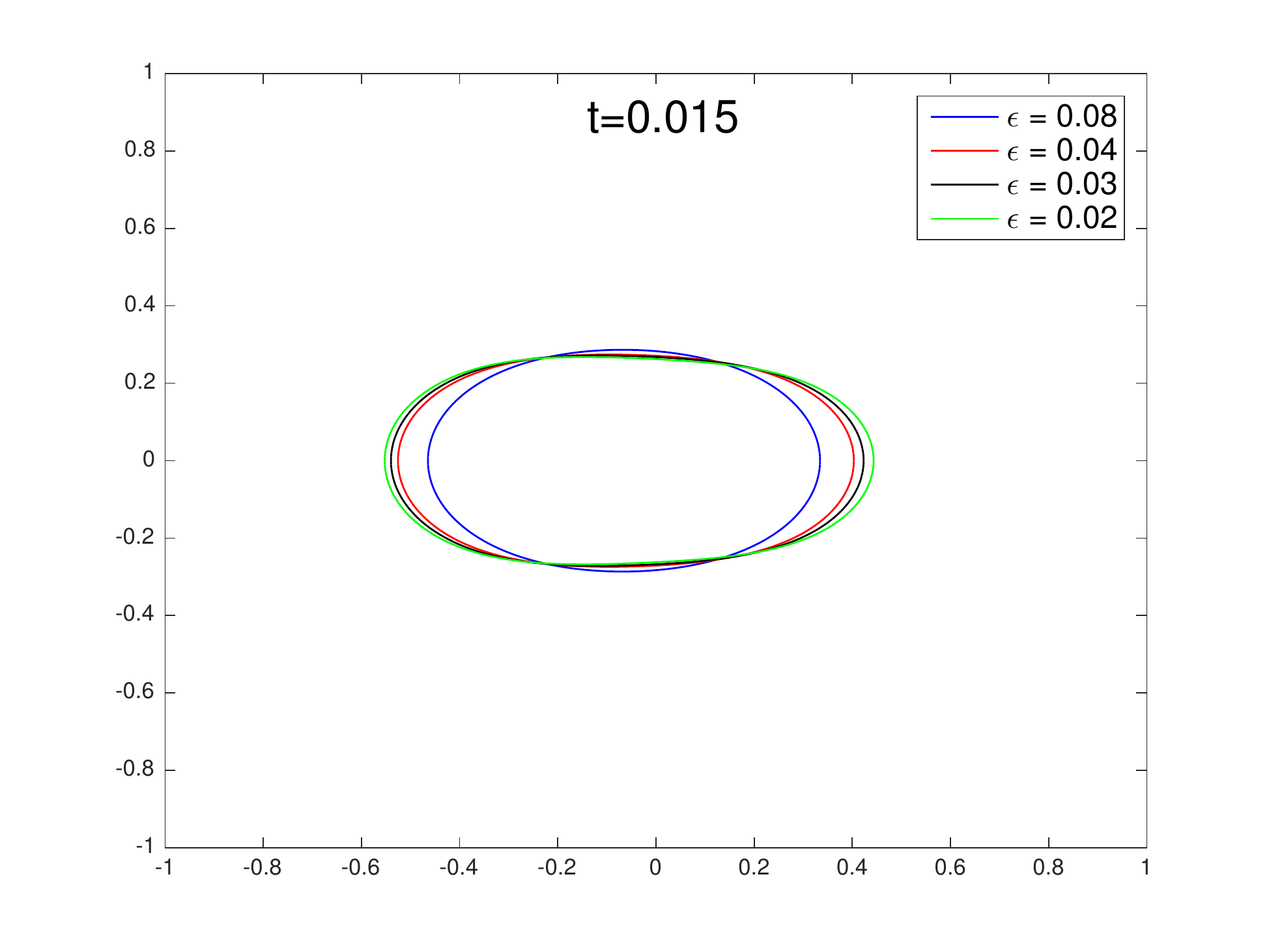} 
\caption{Test 2: Snapshots of the zero-level sets of $u^{\epsilon, k}$
  at $t=0,0.00005,0.0002,0.001, 0.006, 0.015$ and $\epsilon =
    0.08,0.04,0.03,0.02$.}
\label{fig:2circle}
\end{figure}

\section*{Acknowledgements}
The authors Shuonan Wu and Yukun Li highly thank Professor Xiaobing Feng in the University of Tennessee at Knoxville for his motivation for this paper.

\bibliographystyle{siamplain}
\bibliography{Morley_HS} 

\end{document}